\documentclass[12pt, reqno]{amsart}
\usepackage{amssymb,amsthm,amsmath}
\usepackage[numbers,sort&compress]{natbib}
\usepackage{amssymb,amsmath}
\usepackage{amsfonts}
\usepackage{mathrsfs}
\usepackage{latexsym}
\usepackage{amssymb}
\usepackage{amsthm}
\usepackage{indentfirst}
\let\oldsection\section
\renewcommand\section{\setcounter{equation}{0}\oldsection}

\newtheorem{theorem}{Theorem}[section]
\newtheorem{lemma}{Lemma}[section]
\newtheorem{proposition}{Proposition}[section]
\newtheorem{definition}{Definition}[section]

\date{July 17, 2016}

\begin{document}

\title[Primitive equations with horizontal dissipation]{Strong solutions to the 3D primitive equations with only horizontal dissipation: \\ near $H^1$ initial data}

\author{Chongsheng~Cao}
\address[Chongsheng~Cao]{Department of Mathematics, Florida International University, University Park, Miami, FL 33199, USA}
\email{caoc@fiu.edu}

\author{Jinkai~Li}
\address[Jinkai~Li]{Department of Computer Science and Applied Mathematics, Weizmann Institute of Science, Rehovot 76100, Israel}
\email{jklimath@gmail.com}

\author{Edriss~S.~Titi}
\address[Edriss~S.~Titi]{
Department of Mathematics, Texas A\&M University, 3368 TAMU, College Station,
TX 77843-3368, USA. ALSO, Department of Computer Science and Applied Mathematics, Weizmann Institute of Science, Rehovot 76100, Israel.}
\email{titi@math.tamu.edu and edriss.titi@weizmann.ac.il}


\begin{abstract}
In this paper, we consider the initial-boundary value problem of
the three-dimensional primitive equations for oceanic and atmospheric
dynamics with only horizontal viscosity and horizontal diffusivity.
We establish the local, in time, well-posedness of strong solutions,
for any initial data $(v_0, T_0)\in H^1$, by using the local, in
space, type energy estimate. We also establish the global
well-posedness of strong solutions for this system, with any
initial data $(v_0, T_0)\in H^1\cap L^\infty$, such that
$\partial_zv_0\in L^m$, for some $m\in(2,\infty)$, by using the
logarithmic type anisotropic Sobolev inequality and a logarithmic
type Gronwall inequality. This paper
improves the previous results obtained
in [Cao,~C.; Li,~J.; Titi,~E.~S.: \emph{Global well-posedness of the 3D primitive equations with only horizontal viscosity and diffusivity}, Comm.~Pure Appl.~Math., \bf69~\rm(2016), 1492--1531.], where the initial data $(v_0, T_0)$ was assumed
to have $H^2$ regularity.
\end{abstract}

\maketitle

\allowdisplaybreaks

\section{Introduction}\label{sec1}
In the context of the large-scale oceanic and atmospheric dynamics,
an important feature is that the vertical scale ($10$\,-\,$20$ kilometers) is much smaller than the horizontal scales (several thousands of kilometers), and therefore, the aspect ratio, i.e.\,the ratio of the depth (or hight) to the horizontal width, is very small.
Due to this fact, by the scale analysis (see, e.g.,
Pedlosky \cite{PED}), or taking the small aspect ratio limit to the Navier-Stokes equations (see Az\'erad--Guill\'en \cite{AZGU}
and Li--Titi \cite{LITITINSPE,LITITISURVEY} for the mathematical justification of this limit), one obtains the primitive
equations. The primitive equations form a fundamental block in models
for planetary oceanic and atmospheric dynamics, and are wildly used in
the models of the weather prediction, see, e.g., the books by
Haltiner--Williams \cite{HAWI}, Lewandowski \cite{LEWAN}, Majda \cite{MAJDA}, Pedlosky \cite{PED},
Vallis \cite{VALLIS}, Washington--Parkinson \cite{WP} and Zeng \cite{ZENG}. Moreover, in the oceanic and atmospheric dynamics, due to the strong horizontal turbulent mixing,
the horizontal viscosity and diffusivity are much stronger than the vertical viscosity and diffusivity, respectively.

In this paper, we consider the following version of the primitive equations for oceanic and atmospheric dynamics, which have
only horizontal
dissipation, i.e.\,with only horizontal viscosity and horizontal
diffusivity
\begin{eqnarray}
&\partial_tv+(v\cdot\nabla_H)v+w\partial_zv+\nabla_Hp-\Delta_H v+f_0k\times
v=0,\label{1.1}\\
&\partial_zp+T=0,\label{1.2}\\
&\nabla_H\cdot v+\partial_zw=0,\label{1.3}\\
&\partial_tT+(v\cdot\nabla_H)T+w\partial_zT-\Delta_HT=0,\label{1.4}
\end{eqnarray}
where the horizontal velocity $v=(v^1,v^2)$, the vertical velocity
$w$, the temperature
$T$ and the pressure $p$ are the unknowns, and $f_0$ is the
Coriolis parameter. The notations $\nabla_H=(\partial_x,\partial_y)$
and $\Delta_H=\partial_x^2+\partial_y^2$ are the horizontal gradient
and the horizontal Laplacian, respectively. Notably, the above system
has been first studied by the authors in 
\cite{CAOLITITI3}, where
the global existence of strong solutions were established, for arbitrary
initial data with $H^2$ regularity; the aim of the present paper is to
relax the conditions on the initial data, without losing the global
well-posedness of strong solutions.

The mathematical studies of the primitive equations were started
by Lions--Temam--Wang \cite{LTW92A,LTW92B,LTW95}
in the 1990s, where among other issues, global existence of
weak solutions was established; however, the uniqueness of
weak solutions is still an open question, even
for the two-dimensional case. Note that this is different from the
incompressible Navier-Stokes equations, as it is well-known that
the weak solutions to the two-dimensional incompressible
Navier-Stokes equations are unique, see, e.g., Constantin--Foias \cite{CONFOINSBOOK}, Ladyzhenskaya \cite{LADYZHENSKAYA}, Temam \cite{TEMNSBOOK} and more recently Bardos et al.\, \cite{BLNNT} for the uniqueness of weak solutions,  within the class of three-dimensional Leray-Hopf weak solutions, with initial data that are functions of only two  spatial variables.
However, we would like to point out that, though the general uniqueness of weak solutions to
the primitive equations is still unknown, some particular cases
have been solved, see \cite{BGMR03,KPRZ,PTZ09,TACHIM,LITITIUNIQ},
and in particular, it is proved in \cite{LITITIUNIQ} that
weak solutions, with bounded initial data, to the primitive equations are unique, as long as the discontinuity of the initial data is
sufficiently small. Remarkably, different from the three-dimensional
Navier-Stokes equations, global existence and uniqueness of strong
solutions to the three-dimensional primitive equations has already
been known since the breakthrough work by Cao--Titi \cite{CAOTITI2}.
This global existence of strong solutions to the primitive equations
were also proved later by Kobelkov \cite{KOB06} and
Kukavica--Ziane\cite{KZ07A,KZ07B}, by using some different approaches,
see also Hieber--Kashiwabara \cite{HIEKAS} and
Hieber--Hussien--Kashiwabara \cite{HIEHUSKAS} for some generalizations
in the $L^p$ settings, and Coti Zelati et al.\, \cite{COTIZELATI} and Guo--Huang \cite{GH1,GH2} for the
primitive equations coupled with the moisture equations.

Recall that in the oceanic and atmospheric dynamics, due to the
strong horizontal turbulent mixing,
the horizontal viscosity and diffusivity are much stronger than
the vertical viscosity and diffusivity, respectively, and the
vertical viscosity and diffusivity are very weak. While in
all the papers mentioned above, the systems considered are
assumed to have both full viscosity and full diffusivity.
These lead to the studies of the primitive equations with
partial viscosity or partial diffusivity, which have been
carried out by Cao--Titi in \cite{CAOTITI3},
and by Cao--Li--Titi in \cite{CAOLITITI1,CAOLITITI2,CAOLITITI3,CAOLITITI4}, and see
also the survey paper by Li--Titi \cite{LITITISURVEY}.
In particular, the results in \cite{CAOLITITI3,CAOLITITI4} show
that the vertical viscosity is not necessary for
the global existence of strong solutions to the primitive equations,
which is consistent with the physical case (the vertical viscosity
for the large scale atmosphere is weak).
However, on the other hand,
the inviscid primitive equation, with or without
coupling to the heat equation has been shown by Cao et al.\,
\cite{CINT} to blow up in finite time (see also Wong \cite{TKW}).
Combining the global existence results in \cite{CAOLITITI3,CAOLITITI4}
and the finite-time blow up results in \cite{CINT,TKW}, one can conclude
that the horizontal viscosity plays an essential role in stabilizing
the flow in the large-scale atmosphere and ocean. This provides
the mathematical evidences that, in the study of the large scale
atmospheric and oceanic dynamics, one can not ignore the eddy viscosity
in the horizontal direction, created by the strong horizontal turbulent
mixing.

We also note that, for the primitive equations with full viscosity
and full diffusivity, the global existence of strong solutions
are established for any initial data in $H^1$ (see \cite{CAOTITI2,KZ07A,KZ07B,KOB06}), while for the primitive
equations with partial viscosity or partial diffusivity,
caused by the loss partial viscosities or partial diffusivity,
the global strong solutions are established for initial data in
$H^2$ (see \cite{CAOLITITI1,CAOLITITI2,CAOLITITI3}) or some space weaker
than $H^2$ but stronger than $H^1$ (see \cite{CAOLITITI4}).
Compared the results for the primitive equations with partial viscosity
or partial diffusivity \cite{CAOLITITI1,CAOLITITI2,CAOLITITI3,CAOLITITI4} and those with
both full viscosity and full diffusivity \cite{CAOTITI2,KZ07A,KZ07B,KOB06},
one may expect to also establish the $H^1$ theory for the primitive
equations with partial viscosity or partial diffusivity, and this
paper is devoted to some studies in this direction.

In this paper, we continue the study of the primitive equations with
both horizontal viscosity and horizontal diffusivity, which has been
studied in \cite{CAOLITITI3}. The aim of this paper is to improve the
results in \cite{CAOLITITI3}, and in
particular, we want to find the initial data
space as weak as possible to guarantee the global existence of strong
solutions to system (\ref{1.1})--(\ref{1.4}). Recalling that it is the
space $H^1$ that the initial data is taken from to establish the global
existence of strong solutions to the primitive equations with full
dissipation, a natural candidate of the initial data spaces is $H^1$
for the global existence of strong solutions to system
(\ref{1.1})--(\ref{1.4}). As it will be shown in this paper, it is the
case for the local well-posedness, in other words, for any initial data
in $H^1$, there is a unique local strong solution to system
(\ref{1.1})--(\ref{1.4}), subject to some appropriate boundary and
initial conditions; however, due to the lack of the vertical viscosity
and the strongly nonlinear term $w\partial_zv$ (which is eventually
quadratic in
$\nabla v$), the merely $H^1$ regularity is not sufficient for us to
obtain the global strong solutions. Nevertheless, we can prove in this
paper that a slightly better space than $H^1$ is sufficient to
guarantee the global existence of strong solutions. More precisely, we
prove that for any initial data $(v_0, T_0)\in H^1\cap L^\infty$, with
$\partial_z v_0\in L^m$, for some $m\in(2,\infty)$, there is a global
strong solution to system (\ref{1.1})--(\ref{1.4}), subject to some
appropriate boundary conditions.

We consider the problem in the domain $\Omega_0=M\times(-h,0)$, with $M=(0,1)\times(0,1)$, and supplement system (\ref{1.1})--(\ref{1.4}) with the following boundary and initial conditions:
\begin{eqnarray}
& v, w\mbox{ and } T \mbox{ are }\mbox{periodic in }x \mbox{ and }y, \mbox{ and of periods }1,\label{BC}\\
&(\partial_zv,w)|_{z=-h,0}=(0,0),\quad T|_{z=-h}=1,\quad T|_{z=0}=0,\label{BC.1}\\
&(v,T)|_{t=0}=(v_0, T_0). \label{IC}
\end{eqnarray}

System
(\ref{1.1})--(\ref{1.4}) defined on $\Omega_0=M\times(-h,0)$, subject
to the boundary and initial conditions (\ref{BC})--(\ref{IC}), is
equivalent to the following system defined on the extended domain $\Omega:=M\times(-h, h)$
(see, e.g., \cite{CAOLITITI1,CAOLITITI2} for the details)
\begin{eqnarray}
&\partial_tv+(v\cdot\nabla_H)v-\left(\int_{-h}^z\nabla_H\cdot
v(x,y,\xi,t)d\xi\right)\partial_zv-\Delta_Hv\nonumber\\
&~~~\qquad+f_0k\times v+\nabla_H\left(p_s(x,y,t)-\int_{-h}^zT(x,y,\xi,t)d\xi\right)=0,\label{1.5}\\
&\int_{-h}^h\nabla_H\cdot v(x,y,\xi,t)d\xi=0,\label{1.6}\\
&\partial_tT+v\cdot\nabla_HT-\left(\int_{-h}^z\nabla_H\cdot
v(x,y,\xi,t)d\xi\right)\left(\partial_zT+\frac{1}{h}\right)-\Delta_HT=0,\label{1.7}
\end{eqnarray}
subject to the following boundary and initial conditions
\begin{eqnarray}
&v\mbox{ and } T \mbox{ are periodic in }x, y, z,\label{1.8}\\
&v\mbox{ and }T\mbox{ are even and odd in }z,\mbox{ respectively}, \label{1.8-1}\\
&(v,T)|_{t=0}=(v_0, T_0). \label{1.9}
\end{eqnarray}
It should be noticed that in (\ref{1.8}), as well as in all the cases of 
periodic boundary conditions below, the periods in $x,y$ are $1$, while
that in $z$ is $2h$.

Throughout this paper, we use $L^q(\Omega), L^q(M)$ and
$W^{m,q}(\Omega), W^{m,q}(M)$ to denote the standard Lebesgue and
Sobolev spaces, respectively. For $q=2$, we use $H^m$ instead of
$W^{m,2}$. For simplicity, we still use the notations $L^p$ and $H^m$
to denote the $N$ product spaces $(L^p)^N$ and $(H^m)^N$, respectively.
We always use $\|u\|_p$ to denote the $L^p$ norm of $u$. We denote by
$\textbf{x}^\text{H}$ a point in $\mathbb R^2$. For $0<r<\infty$, we
use $D_r(\textbf{x}^\text{H})$ to denote an open disk in $\mathbb R^2$ of radius $r$ centered
at $\textbf{x}^\text{H}$. We always use $D_r$ to stand for the disk
centered at the origin, except when otherwise explicity mentioned.

\begin{definition}\label{def1.1}
Given a positive number $\mathcal T$. Let $v_0, T_0\in H^1(\Omega)$ be
two spatially
periodic functions, such that they are even and odd in $z$,
respectively. A pair $(v,T)$ is called a strong solution to system
(\ref{1.5})--(\ref{1.9}) on $\Omega\times(0,\mathcal T)$ if

(i) $v$ and $T$ are spatially periodic, and they are even and odd in $z$, respectively;

(ii) $v$ and $T$ have the following regularity properties
\begin{eqnarray*}
&&v,T\in L^\infty(0,\mathcal T; H^1(\Omega))\cap C([0,\mathcal T];L^2(\Omega)),\\
&&\nabla_Hv,\nabla_HT\in L^2(0,\mathcal T; H^1(\Omega)),\quad\partial_tv,\partial_tT\in L^2(\Omega\times(0,\mathcal T));
\end{eqnarray*}

(iii) $v$ and $T$ satisfy equations (\ref{1.5})--(\ref{1.7}) a.e.\,in $\Omega\times(0,\mathcal T)$ and the initial condition (\ref{1.9}).
\end{definition}

\begin{definition}
The pair $(v,T)$ is called a global strong solution to system (\ref{1.5})--(\ref{1.9}) if it is a strong solution on $\Omega\times(0,\mathcal T)$, for any $\mathcal T\in(0,\infty)$.
\end{definition}

Our main result is concerning the local and global well-posedness of strong solutions to system (\ref{1.5})--(\ref{1.7}), subject to (\ref{1.8})--(\ref{1.9}), as stated in the following:

\begin{theorem}\label{thm1}
Suppose that the periodic functions $v_0,T_0\in H^1(\Omega)$ are
even and odd in $z$, respectively, with $\int_{-h}^h\nabla_H\cdot
v_0(x,y,z)dz=0$, for any $(x,y)\in M$.  Then, there is a unique local,
in time, strong solution $(v,T)$ to system
(\ref{1.5})--(\ref{1.7}), subject to
the boundary and initial conditions (\ref{1.8})--(\ref{1.9}).

Moreover, if we assume in addition that
$$
\partial_zv_0\in L^m(\Omega),\quad (v_0, T_0)\in L^\infty(\Omega),
$$
for some $m\in(2,\infty)$, then the corresponding local strong solution $(v,T)$ can be extended uniquely to be a global one.
\end{theorem}

Since we consider the system with only
horizontal dissipation and the initial data is taken to belong only to $H^1$,
the arguments used in \cite{GMR01} (with $H^1$ initial data but for
full
dissipation case) and \cite{CAOLITITI1,CAOLITITI2,CAOLITITI3,CAOLITITI4}
(for partial dissipation case but with $H^2$ initial data) do not apply
here in order to show the short time existence.
Actually, as it will be explained below, applying the standard energy
approach to system (\ref{1.5})--(\ref{1.7}) does not yield the required $H^1$
estimates, even locally in time.
The crucial step to prove the local existence of strong solutions is to
obtain a local in time estimate for
$\partial_zv$. One may try to use the standard energy approach to get
such an estimate; however, by doing that, one will encounter a term
on the right-hand side of the energy inequality which can not be
controlled by the quantities on the left-hand side, unless we have some
additional smallness conditions.
To overcome this difficulty, we employ a local in space type energy
inequality instead of the global in space type.
The key idea of the local in space
type energy inequality is that initially the local in space integral norms can be as small as desired, which is guaranteed
by the absolute continuity of integrals, and we can expect that they
will remain small for a short time. As we will see in Proposition
\ref{lem3.2}, we can successfully achieve the expected estimate for
$\partial_zv$ by using the local in space type energy inequality. Moreover, based
on this estimate, we can obtain other relevant
estimates which are sufficient to
prove the local, in time, existence of strong solutions.

To prove the global existence of strong solutions, we adopt the ideas
employed in \cite{CAOLITITI3,CAOLITITI4}. The key issue is to derive
an $L^\infty(0,\mathcal T; L^2(\Omega))$ estimate for $\partial_zv$ for
any positive finite time $\mathcal T$. To this end, due to the
absence of the vertical
viscosity in the momentum equations, we have to derive some control on
$\|v\|_\infty^2$, which appears as a factor in the
energy inequalities of the form (see Proposition \ref{prop5.2})
$$
\frac{d}{dt}A(t)+B(t)\leq C\|v\|_\infty^2A(t)+\mbox{other terms},
$$
where $A$ involves $\|\partial_zv\|_2$.
The treatment on $\|v\|_\infty^2$ is similar to that in
\cite{CAOLITITI3,CAOLITITI4}. More precisely, thanks to the estimates
on the growth of the $L^q$ norms of $v$ (Proposition \ref{prop5.1})
and the logarithmic type Sobolev embedding inequality
(Lemma \ref{logsob}), such $\|v\|_\infty^2$ can be controlled
by $\log(A(t)+B(t))$, and as a result, by the logarithmic type
Gronwall inequality (Lemma \ref{loggron}), we can obtain the
desired estimate.

The rest of this paper is arranged as follows: in the next section,
section \ref{sec2}, we collect some preliminary results which will be
used in the subsequent sections; in section \ref{sec3}, we prove the
local well-posedness part of Theorem \ref{thm1}, where the local in
space energy inequality will be employed; in section \ref{sec4}, we
prove the global existence part of Theorem \ref{thm1}, where the
logarithmic type anisotropic Sobolev inequality and the logarithmic type
Gronwall inequality will be employed; in the last Appendix
section, we give the
proof of a logarithmic type anisotropic Sobolev inequality,
a generalization of the Brezis-Gallouet-Wainger
inequality \cite{Brezis_Gallouet_1980,Brezis_Wainger_1980}.

\section{Preliminaries}\label{sec2}

In this section, we collect some preliminary results which will be used in the rest of this paper.

\begin{lemma}\label{lad}
Let $S$ be a bounded domain in $\mathbb R^2$, and $\phi,\varphi,\psi$
are functions defined on $S\times(-h,h)$. Denote by $L$ the diameter
of $S$. Then, it holds that
\begin{align}
  &\int_{S}\left(\int_{-h}^h|\phi(x,y,z)|dz\right)\left(\int_{-h}^h|\varphi(x,y,z)\psi
  (x,y,z)|dz\right)dxdy\nonumber\\
  \leq&Ch^{\frac12}\min\Big\{\|\phi\|_2  \|\varphi\|_2^{1/2}\left(\frac{\|\varphi\|_2}{L} +\|\nabla_H\varphi\|_2\right)^{\frac12}
  \|\psi\|_2^{\frac12}\left(\frac{\|\psi\|_2}{L} +\|\nabla_H\psi\|_2\right)^{\frac12},\nonumber\\ &\|\phi\|_2^{1/2}\left(\frac{\|\phi\|_2}{L} +\|\nabla_H\phi\|_2\right)^{\frac12}
  \|\varphi\|_2^{1/2}\left(\frac{\|\varphi\|_2}{L} +\|\nabla_H\varphi\|_2\right)^{\frac12}\label{n2.1}
  \|\psi\|_2 \Big\},
  \end{align}
  \begin{align}
  &\int_{S}\left(\int_{-h}^h|\phi |dz\right)\left(\int_{-h}^h|\varphi \psi
  |dz\right)dxdy\nonumber\\
  \leq&Ch^{\frac56} \|\phi\|_6  \|\varphi\|_2^{\frac23}\left(\frac{\|\varphi\|_2}{L} +\|\nabla_H\varphi\|_2\right)^{\frac13}
  \|\psi\|_2 ,\label{n2.2}
  \end{align}
  and
  \begin{align}
  &\int_{S}\left(\int_{-h}^h|\phi |^2dz\right)\left(\int_{-h}^h|\varphi |^2dz
  \right)dxdy\nonumber\\
  \leq&C\|\phi\|_2\left(\frac{\|\phi\|_2}{L}+\|\nabla_H\phi\|_2\right) \|\varphi\|_2\left(\frac{\|\varphi\|_2}{L}+\|\nabla_H\varphi\|_2\right),
  \label{n2.3}
  \end{align}
where $C$ is a constant depending only on the shape of $S$, but not on
its size.
\end{lemma}

\begin{proof}
Note that (\ref{n2.1}) has been included in Lemma 2.1 in
\cite{CAOLITITI3}. We now consider the proof of (\ref{n2.2}).
By the H\"older, (dimensionless) Gagliardo-Nirenberg and Minkowski inequalities, we deduce
\begin{align*}
  &\int_{S}\left(\int_{-h}^h|\phi(x,y,z)|dz\right)\left(\int_{-h}^h|\varphi(x,y,z)\psi
  (x,y,z)|dz\right)dxdy\\
  \leq&\int_S\left(\int_{-h}^h|\phi|dz\right)\left(\int_{-h}^h|\varphi|^2 dz\right)^{1/2}
  \left(\int_{-h}^h|\psi|^2dz\right)^{1/2}dxdy\\
  \leq&\left[\int_S\left(\int_{-h}^h|\phi|dz\right)^6dxdy\right]^{1/6}
  \left[\int_S\left(\int_{-h}^h|\varphi|^2dz\right)^{3/2}dxdy\right]^{1/3}\|\psi\|_2\\
  \leq&Ch^{\frac56}\|\phi\|_6\left[\int_{-h}^h\left(\int_S|\varphi|^3dxdy\right)^{2/3}dz \right]^{1/2}\|\psi\|_2\\
  \leq&Ch^{\frac56}\|\phi\|_6\left(\int_{-h}^h\|\varphi\|_{2,S}^{4/3} \left(\frac{\|\varphi\|_{2,S}}{L}+\|\nabla_H\varphi\|_{2,S}\right) ^{2/3}dz\right)^{1/2}
  \|\psi\|_2\\
  \leq&Ch^{\frac56}\|\phi\|_6  \|\varphi\|_2^{2/3}\left(\frac{\|\varphi\|_2}{L} +\|\nabla_H\varphi\|_2\right)^{1/3}
  \|\psi\|_2,
\end{align*}
proving (\ref{n2.2}). Similarly
\begin{align*}
  &\int_{S}\left(\int_{-h}^h|\phi(x,y,z)|^2dz\right)\left(\int_{-h}^h|\varphi(x,y,z)|^2dz
  \right)dxdy\nonumber\\
  \leq&C\left[\int_S\left(\int_{-h}^h|\phi|^2dz\right)^2dxdy\right]^{1/2}
  \left[\int_S\left(\int_{-h}^h|\varphi|^2dz\right)^2dxdy\right]^{1/2}\\
  \leq&C\left[\int_{-h}^h\left(\int_S|\phi|^4dxdy\right)^{1/2}dz\right]
  \left[\int_{-h}^h\left(\int_S|\varphi|^4dxdy\right)^{1/2}dz\right]\\
  \leq&C \int_{-h}^h\|\phi\|_{2,S}\left(\frac{\|\phi\|_{2,S}}{L}+ \|\nabla_H\phi\|_{2,S}\right)dz \\
  &\times \int_{-h}^h\|\varphi\|_{2,S}\left(\frac{\|\varphi\|_{2,S}}{L} +\|\nabla_H\varphi\|_{2,S}\right)dz \\
  \leq&C\|\phi\|_2\left(\frac{\|\phi\|_2}{L}+\|\nabla_H\phi\|_2\right) \|\varphi\|_2\left(\frac{\|\varphi\|_2}{L}+\|\nabla_H\varphi\|_2\right),
\end{align*}
proving (\ref{n2.3}).
\end{proof}

As a directly corollary of Lemma \ref{lad}, noticing that all discs have the same shape,
we have the following lemma.

\begin{lemma}
  \label{lads}
  Let $D_r$ be an arbitrary disk of radius $r$ in $\mathbb R^2$, then we have
\begin{align*}
  &\int_{D_r}\left(\int_{-h}^h|\phi |dz\right)\left(\int_{-h}^h|\varphi \psi
  |dz\right)dxdy\\
  \leq&Ch^\frac12\min\Bigg\{\|\phi\|_2  \|\varphi\|_2^{\frac12} \left(\frac{\|\varphi\|_2}{r}+\|\nabla_H\varphi\|_2\right)^{\frac12}
  \|\psi\|_2^{\frac12}\left(\frac{\|\psi\|_2}{r} +\|\nabla_H\psi\|_2\right)^{\frac12},\\
  &\|\phi\|_2^{\frac12}\left(\frac{\|\phi\|_2}{r} +\|\nabla_H\phi\|_2\right)^{\frac12}
  \|\varphi\|_2^{\frac12}\left(\frac{\|\varphi\|_2}{r} +\|\nabla_H\varphi\|_2\right)^{\frac12}
  \|\psi\|_2 \Bigg\},
  \end{align*}
  and
  \begin{align*}
  &\int_{D_r}\left(\int_{-h}^h|\phi |dz\right)\left(\int_{-h}^h|\varphi \psi
   |dz\right)dxdy\\
  \leq& C h^{\frac56}\|\phi\|_6  \|\varphi\|_2^{\frac23}\left(\frac{\|\varphi\|_2}{r} +\|\nabla_H\varphi\|_2\right)^{\frac13}
  \|\psi\|_2,
\end{align*}
for an absolute constant $C$.
\end{lemma}
The following lemma is a Sobolev embedding inequality in terms of
mixed norm $L^p$ spaces, see Li--Xin \cite{LIXIN} for a similar result.

\begin{lemma}\label{sob}Let $S$ be a bounded subset of $\mathbb R^2$.
Denote by $L$ the diameter of $S$. Then, for any function $f$ defined
on $S\times(-h,h)$, it holds that
\begin{eqnarray*}
&&\sup_{-h\leq z\leq h}\|f(\cdot,z)\|_{2,S}\leq \|f\|_2^{1/2}\left(\frac{\|f\|_2}{2h}+2\|\partial_zf\|_2\right)^{1/2},\\
&&\sup_{-h\leq z\leq h}\|f(\cdot,z)\|_{4,S}\leq C\left(\frac{\|f\|_2}{h}+\|\partial_zf\|_2\right)^{1/2} \left(\frac{\|f\|_2}{L}+\|\nabla_Hf\|_2\right)^{1/2},
\end{eqnarray*}
where $C$ is a positive constant depending only on the shape of $S$.
In particular, if $S=D_r$, an arbitrary disk of radius $r$ in $\mathbb R^2$, then one has
$$
\sup_{-h\leq z\leq h}\|f(\cdot,z)\|_{L^4(D_r)}\leq C\left(\frac{\|f\|_2}{r}+\|\nabla_Hf\|_2\right)^{1/2} \left(\frac{\|f\|_2}{h}+\|\partial_zf\|_2\right)^{1/2},
$$
for an absolute positive constant $C$.
\end{lemma}

\begin{proof}
Using the fact that $|k(z)|\leq\frac{1}{2h}\int_{-h}^h|k(\xi)|d\xi+\int_{-h}^h|k'(\xi)|d\xi$, for any $z\in(-h,h)$, one has
\begin{eqnarray}
\sup_{-h\leq z\leq h}\|f(\cdot,z)\|_{2,S}^2\leq\frac{1}{2h}\int_{h}^h\|f\|_{2,S}^2dz+\int_{-h}^h\left|\frac{d}{dz}\|f\|_{2,S}^2\right|dz\nonumber\\
\leq\frac{\|f\|_2^2}{2h}+2\int_\Omega|f||\partial_zf|dxdydz\leq \|f\|_2\left(\frac{\|f\|_2}{2h}+2\|\partial_zf\|_2\right),\label{2.2}
\end{eqnarray}
proving the first conclusion.

Using again the fact that $|k(z)|\leq\frac{1}{2h}\int_{-h}^h|k(\xi)|d\xi+\int_{-h}^h|k'(\xi)|d\xi$, for all $z\in(-h,h)$, it follows from the (dimensionless) Gagliardo-Nirenberg and Cauchy-Schwarz inequalities that
\begin{align*}
  &\sup_{-h\leq z\leq h}\|f(\cdot,z)\|_{4,S}^4\leq\frac{1}{2h}\int_{-h}^h\|f\|_{4,S}^4dz+\int_{-h}^h\left|\frac{d}{dz}\|f\|_{4,S}^4\right|dz\\
  \leq&\frac{\|f\|_4^4}{2h}+4\int_\Omega|f|^3|\partial_zf|dxdydz
  \leq \frac{\|f\|_4^4}{2h} +4\int_{-h}^h\|f\|_{6,S}^3\|\partial_zf\|_{2,S}dz\\
  \leq&\frac{\|f\|_4^4}{2h} +C\int_{-h}^h\|f\|_{4,S}^2\left(\frac{\|f\|_{2,S}}{L}+\|\nabla_H f\|_{2,S}\right)\|\partial_zf\|_{2,S}dz\\
  \leq&\frac{\|f\|_4^4}{2h}+C\left(\sup_{-h\leq z\leq h}\|f\|_{4,S}^2\right)\int_{-h}^h\left(\frac{\|f\|_{2,S}}{L} +\|\nabla_H f\|_{2,S}\right)\|\partial_zf\|_{2,S}dz\\
  \leq&\frac{1}{2}\sup_{-h\leq z\leq h}\|f\|_{4,S}^4+C\left[\frac{\|f\|_4^4}{h} +\left(\frac{\|f\|_2^2}{L^2}+\|\nabla_Hf\|_2^2\right) \|\partial_zf\|_2^2\right],
\end{align*}
and thus
\begin{equation}
  \sup_{-h\leq z\leq h}\|f(\cdot,z)\|_{4,S}^4\leq C\left[\frac{\|f\|_4^4}{h} +\left(\frac{\|f\|_2^2}{L^2}+\|\nabla_Hf\|_2^2\right) \|\partial_zf\|_2^2\right]. \label{2.3}
\end{equation}

We estimate $\|f\|_4^4$, on the right-hand side of the above inequality, as follows. By the (dimensionless) two-dimensional Ladyzhenskaya inequality, it follows from (\ref{2.2}) that
\begin{align*}
  \|f\|_4^4=&\int_{-h}^h\|f\|_{4,S}^4dz\leq C\int_{-h}^h\|f\|_{2,S}^2\left(\frac{\|f\|_{2,S}}{L}+ \|\nabla_H f\|_{2,S}\right)^2 dz\\
  \leq&C\left(\sup_{-h\leq z\leq h}\|f\|_{2,S}^2\right) \left(\frac{\|f\|_2^2}{L^2}+\|\nabla_Hf\|_2^2\right)\\
  \leq&C\|f\|_2\left(\frac{\|f\|_2}{h}+\|\partial_zf\|_2\right) \left(\frac{\|f\|_2^2}{L^2}+\|\nabla_Hf\|_2^2\right).
\end{align*}
Substituting the above inequality into (\ref{2.3}) and using the Cauchy-Schwarz inequality yield
\begin{align*}
\sup_{-h\leq z\leq h}\|f(\cdot,z)\|_{4,S}^4\leq & C\left[\|\partial_zf\|_2^2+ \frac{\|f\|_2}{h}\left(\frac{\|f\|_2}{h}+\|\partial_zf\|_2\right)\right] \left(\frac{\|f\|_2^2}{L^2}+\|\nabla_Hf\|_2^2\right)\\
\leq&C\left(\frac{\|f\|_2^2}{h^2}+\|\partial_zf\|_2^2\right) \left(\frac{\|f\|_2^2}{L^2}+\|\nabla_Hf\|_2^2\right).
\end{align*}
which implies
$$
\sup_{-h\leq z\leq h}\|f(\cdot,z)\|_{4,S}\leq C\left(\frac{\|f\|_2}{h}+\|\partial_zf\|_2\right)^{1/2} \left(\frac{\|f\|_2}{L}+\|\nabla_Hf\|_2\right)^{1/2},
$$
proving the second conclusion. The third conclusion is a straightforward corollary of the second one, as all the discs have the same shape. This completes the proof.
\end{proof}

We also need the following logarithmic type anisotropic Sobolev
embedding inequality, which generalizes that in \cite{CAOLITITI3}
to the anisotropic case, see Cao--Farhat--Titi \cite{CAOFARHATTITI}, Cao--Wu \cite{CAOWU} and Danchin--Paicu \cite{DANCHINPAICU} for some
related inequalities in the isotropic setting in the 2D case.
Some similar anisotropic inequality was used by Li--Titi \cite{LITITIBE} in the study of the
Boussinesq equations to relax the assumptions on the initial data.

\begin{lemma}
  \label{logsob}
Let $\textbf{p}=(p_1, p_2, p_3)$, with $p_i\in(1,\infty)$, and
$\frac{1}{p_1}+\frac{1}{p_2}+\frac{1}{p_3}<1.$
Then, for any periodic function $F$ on $\Omega$, we have
\begin{equation*}
  \|F\|_\infty\leq C_{ {\textbf p},\lambda,\Omega}\max\left\{1,\sup_{r\geq2}\frac{\|F\|_r}{r^\lambda}\right\}
  \log^\lambda
  \left(\sum_{i=1}^3(\|F\|_{p_i}+\|\partial_iF\|_{p_i})+e\right),
\end{equation*}
for any $\lambda>0$.
\end{lemma}

\begin{proof}
Extending $F$ periodically to the whole space. Take a function $\phi\in C_0^\infty(\mathbb R^3)$, such that $\phi\equiv1$ on $\Omega$, and $0\leq\phi\leq1$ on $\mathbb R^3$. Set $f=F\phi$.
Noticing that
$$
\|F\|_\infty\leq\|f\|_{L^\infty(\mathbb R^3)}, \quad\|f\|_{L^r(\mathbb R^3)}\leq C\|F\|_r,\quad\|\partial_if\|_{L^{r}(\mathbb R^3)}\leq C(\|F\|_r+\|\partial_iF\|_r),
$$
it follows from Lemma \ref{lemapp} (see the Appendix) that
\begin{align*}
  &\|F\|_\infty\leq\|f\|_{L^\infty(\mathbb R^3)}\\
  \leq&C_{\textbf{p},\lambda}\max\left\{1,\sup_{r\geq2} \frac{\|f\|_{L^r(\mathbb R^3)}}{r^\lambda}\right\}
   \log^\lambda
  \left(\sum_{i=1}^3(\|f\|_{L^{p_i}(\mathbb R^3)}+\|\partial_if\|_{L^{p_i}(\mathbb R^3)})+e\right)\\
  \leq&C_{\textbf{p},\lambda,\Omega}\max\left\{1,\sup_{r\geq2} \frac{\|F\|_r}{r^\lambda}\right\}
   \log^\lambda
  \left(\sum_{i=1}^3(\|F\|_{{p_i}}+\|\partial_iF\|_{{p_i}})+e\right),
\end{align*}
proving the conclusion.
\end{proof}

We will use the following logarithmic type Gronwall inequality, see
Li--Titi \cite{LITITIBE,LITITITCM,LITITITCMMOISTURE} for some similar type
Gronwall inequalities with logarithmic terms.

\begin{lemma}
  \label{loggron}
Given a positive time $\mathcal T\in(0,\infty)$. Let $A(t), B(t)$ and $f(t)$ be nonnegative integrable functions on $[0,\mathcal T)$, with $A$ being absolutely continuous on $[0,\mathcal T)$, such that the following holds
$$
A'(t)+B(t)\leq KA(t)\log B(t)+f(t), \quad t\in(0,\mathcal T),
$$
where $K\geq1$ is a constant.
Then, the following estimate holds
$$
A(t)+\int_0^tB(s)ds\leq e^{Q(t)}(1+2Q(t)),
$$
for any $t\in[0,\mathcal T)$, where
$$
Q(t)=e^{Kt}
\left(\log(A(0)+1)+(2K^2+1)t+\int_0^tf(s)ds\right).
$$
\end{lemma}

\begin{proof}
Setting
$$
A_1(t)=A(t)+1,\quad B_1(t)=B(t)+A(t)+1,
$$
then, by assumption, we have
$$
A_1'(t)+B_1(t)\leq KA_1(t)\log B_1(t)+A_1(t)+f(t).
$$
Dividing both sides of the above inequality by $A_1(t)$, noticing that $A_1(t)\geq1$, one has
\begin{align*}
A_1'(t)+\frac{B_1(t)}{A_1(t)}\leq& K\log B_1(t)+f(t)+1\\
=&K\log\frac{B_1(t)}{A_1(t)}+K\log A_1(T)+f(t)+1.
\end{align*}
One can easily check that $\log z\leq 2\sqrt z$, for any $z\geq1$, and thus, noticing that $\frac{B_1(t)}{A_1(t)}\geq1$, we have $\log\frac{B_1(t)}{A_1(t)}\leq 2\sqrt{\frac{B_1(t)}{A_1(t)}}$. Thanks to this estimate, it follows from the above inequality and the Cauchy-Schwarz inequality that
\begin{align*}
A_1'(t)+\frac{B_1(t)}{A_1(t)}\leq&2K\sqrt{\frac{B_1(t)}{A_1(t)}}+K\log A_1(T)+f(t)+1\\
\leq&\frac{B_1(t)}{2A_1(t)}+K\log A_1(T)+f(t)+2K^2+1,
\end{align*}
that is
$$
A_1'(t)+\frac{B_1(t)}{2A_1(t)}\leq K\log A_1(t)+f(t)+2K^2+1.
$$
Applying the Gronwall inequality to the above inequality yields
\begin{align*}
\log  A_1(t)+\int_0^t\frac{B_1(s)}{2A_1(s)}ds\leq & e^{Kt}
\left(\log A_1(0)+(2K^2+1)t+\int_0^tf(s)ds\right)\\
=& e^{Kt}
\left(\log(A(0)+1)+(2K^2+1)t+\int_0^tf(s)ds\right)=:Q(t).
\end{align*}
It is obvious that $Q(t)$ is an increasing function of $t$. Therefore, we deduce
\begin{align*}
A_1(t)+\int_0^tB_1(s)ds=&e^{\log A_1(t)}+\int_0^t\frac{B_1(s)}{2A_1(s)} 2e^{\log A_1(s)}ds\\
  \leq&e^{Q(t)}+2\int_0^t\frac{B_1(s)}{2A_1(s)}e^{Q(s)}ds\leq e^{Q(t)}(1+2Q(t)),
\end{align*}
which, recalling the definitions of $A_1$ and $B_1$, implies the conclusion.
\end{proof}

The next lemma is a version of the Aubin-Lions lemma.

\begin{lemma}\label{AL}
[See Simon \cite{Simon} Corollary 4] Let $\mathcal T\in(0,\infty)$. Assume that $X, B$ and $Y$ are three Banach spaces, with
$X\hookrightarrow\hookrightarrow B\hookrightarrow Y.$ Then it holds that

(i) If $\mathscr F$ is a bounded subset of $L^p(0, \mathcal T; X)$, where $1\leq p<\infty$, and $\frac{\partial \mathscr F}{\partial t}=\left\{\frac{\partial f}{\partial t}|f\in \mathscr F\right\}$
is bounded in $L^1(0,\mathcal  T; Y)$, then $\mathscr F$ is relatively compact in $L^p(0,\mathcal  T; B)$;

(ii) If $\mathscr F$ is bounded in $L^\infty(0,\mathcal  T; X)$ and $\frac{\partial \mathscr F}{\partial t}$ is bounded in $L^r(0,\mathcal  T; Y)$, where $r>1$, then $\mathscr F$ is relatively compact in $C([0, \mathcal T]; B)$.
\end{lemma}

The following proposition, which states the global existence of strong solutions to system with full viscosities, can be proven in the same way as in \cite{CAOLITITI2}.

\begin{proposition}\label{glo}
Let $v_0, T_0\in H^2(\Omega)$ be two periodic functions, such that
they are even and odd in $z$, respectively. Then, for any given
$\varepsilon>0$, there is a unique global strong solution $(v,T)$
to the following system
\begin{eqnarray}
&\partial_tv+(v\cdot\nabla_H)v-\left(\int_{-h}^z\nabla_H\cdot
v(x,y,\xi,t)d\xi\right)\partial_zv-\Delta_Hv-\varepsilon\partial_z^2v+f_0k\times v\nonumber\\
&+\nabla_H\left(p_s(x,y,t)-\int_{-h}^zT(x,y,\xi,t)d\xi\right)=0,\label{3.1}\\
&\int_{-h}^h\nabla_H\cdot v(x,y,\xi,t)d\xi=0,\label{3.2}\\
&\partial_tT+v\cdot\nabla_HT-\left(\int_{-h}^z\nabla_H\cdot
v(x,y,\xi,t)d\xi\right)\left(\partial_zT+\frac{1}{h}\right)-\Delta_HT -\varepsilon\partial_z^2T=0,\label{3.3}
\end{eqnarray}
subject to the boundary and initial conditions (\ref{1.8})--(\ref{1.9}), such that
\begin{eqnarray*}
&v\in L^\infty_{\text{loc}}([0,\infty); H^2(\Omega))\cap C([0,\infty);H^1(\Omega))\cap L^2_{\text{loc}}([0,\infty);H^3(\Omega)),\\
&T\in L^\infty_{\text{loc}}([0,\infty); H^2(\Omega))\cap C([0,\infty);H^1(\Omega)),\quad\nabla_HT\in L^2_{\text{loc}}([0,\infty);H^2(\Omega)),\\
&\partial_tv,\partial_tT\in L^2_{\text{loc}}([0,\infty);H^1(\Omega)).
\end{eqnarray*}
\end{proposition}

\section{Local well-posedness of strong solutions}\label{sec3}

In this section, we prove the local well-posedness part of
Theorem \ref{thm1}. By Proposition \ref{glo}, for any given
$\varepsilon>0$, system (\ref{3.1})--(\ref{3.3}), subject to the
boundary and initial conditions (\ref{1.8})--(\ref{1.9}), has a unique
global strong solution $(v,T)$. Next, we are going to establish
uniform in $\varepsilon$ estimates of this solution.

\begin{proposition}
  \label{lem3.1}
Let $(v,T)$ be as in Proposition \ref{glo}. Then the following holds
  $$
  \sup_{0\leq s\leq t}(\|v\|_6^2+\|T\|_6^2)(s)+\int_0^t(\|\nabla_Hv\|_2^2 +\varepsilon\|\partial_zv\|_2^2+\|\nabla_H
  T\|_2^2)(s)ds\leq K_1(t),
  $$
  for any $t\in[0,\infty)$, where $K_1$ is a continuous nondecreasing
  function on $[0,\infty)$, which depends on $\|(v_0,T_0)\|_6$ in a
  continuous manner, and is independent of $\varepsilon$.
\end{proposition}

\begin{proof}
This is a direct corollary of Proposition 3.1 in \cite{CAOLITITI3}.
\end{proof}

Suppose $(v,T)$ be as in Proposition \ref{glo}, and set $u=\partial_zv$.  Then $u$ satisfies the equation
\begin{eqnarray}
  &\partial_tu+(v\cdot\nabla_H)u-\left(\int_{-h}^z\nabla_H\cdot v(x,y,\xi,t)d\xi\right)\partial_zu-\Delta_Hu-\varepsilon\partial_z^2u\nonumber\\
  &+f_0k\times u+(u\cdot\nabla_H)v-(\nabla_H\cdot v)u-\nabla_HT=0,\label{3.4}
\end{eqnarray}
in $\Omega\times(0,\infty)$.

The following proposition, which gives the estimates on the
vertical derivative of the velocity, $u$,
plays the key role in proving the
local existence of strong solutions to system (\ref{1.5})--(\ref{1.9}). The basic idea of proving this
proposition is the local, in space, energy inequality.

\begin{proposition}
  \label{lem3.2}
  Let $(v,T)$ be the unique global strong solution to system (\ref{3.1})--(\ref{3.3}), subject to the boundary and initial conditions (\ref{1.8})--(\ref{1.9}).

  There is a suitably small positive constant $\delta_0\leq1$, depending only on $h$, such that if
  $$
  \sup_{\textbf{x}^\text{H}\in M}\int_{-h}^h\int_{D_{2r_0}(\textbf{x}^\text{H})}|u_0(x,y,z)|^2dxdydz\leq\delta_0^2,
  $$
  for some positive number $r_0\leq1$, then the following holds true
  $$
  \sup_{0\le t\leq t_0^*}\|u\|_2^2(t)+\int_0^{t_0^*} (\|\nabla_Hu\|_2^2+\varepsilon\|\partial_zu\|_2^2)(t)dt\leq C,
  $$
  where $C$ is a constant, depending only on $\delta_0$ and $r_0$, and where $t_0^*=\min\left\{1,\frac{6r_0^2\delta_0^2}{C_0}\right\}$, with a positive constant $C_0$ depending only on $h, \delta_0$ and $r_0$.
\end{proposition}

\begin{proof}
Since $u,v$ and $T$ are periodic, they are defined on the whole
space, consequently, equation (\ref{3.4}) is satisfied on the whole space.
For any $\textbf{x}^\text{H}\in M$, set
$Q_r(\textbf{x}^\text{H})=D_r(\textbf{x}^\text{H})\times(-h,h)$. If
$\textbf{x}^\text{H}$ is the original point, we simply use the notation
$Q_r$
instead of $Q_r(0)$. Let $\delta_0\leq1$ be a sufficiently small
positive number, to be
determined later. Let $r_0\leq1$ be a small enough
positive number such that
$$
\sup_{\textbf{x}^\text{H}\in M}\int_{-h}^h\int_{D_{2r_0}(\textbf{x}^\text{H})}|u_0(x,y,z)|^2dxdydz\leq\delta_0^2.
$$
Set
$$
t_0=\sup\left\{t\in(0,1]\left|\sup_{0\leq s\leq t}\sup_{\textbf{x}^\text{H}\in M}\int_{-h}^h\int_{D_{r_0}(\textbf{x}^\text{H})}
|u(x,y,z,s)|^2dxdydz\leq8\delta_0^2\right.\right\}.
$$
Since any disk of radius $2r_0$ can be covered by finite many, say $N$, which is independent of $r_0$, disks of radius $r_0$, one has
\begin{equation}
  \sup_{0\leq t\leq t_0}\sup_{\textbf{x}^\text{H}\in M}\int_{-h}^h\int_{D_{2r_0}(\textbf{x}^\text{H})}
|u(x,y,z)|^2dxdydz\leq 8N\delta_0^2. \label{1}
\end{equation}

Consider a cut-off function $\varphi\in C_0^\infty((D_{2r_0}))$,
such that $0\leq\varphi\leq1,  |\nabla_H\varphi|\leq\frac{C}{r_0}$ on
$D_{2r_0}$, with an absolute constant $C$, and $\varphi\equiv1$ on
$D_{r_0}$. Multiplying equation (\ref{3.4}) by $u\varphi^2$ and
integrating over $Q_{2r_0}$, then it follows from
integrating by parts that
\begin{align}
\frac{1}{2}\frac{d}{dt}\int_{Q_{2r_0}}& |u|^2\varphi^2dxdydz+\int_{Q_{2r_0}}(\nabla_Hu:\nabla_H(u\varphi^2)+
  \varepsilon\partial_zu\partial_z(u\varphi^2))dxdydz\nonumber\\
  =&-\int_{Q_{2r_0}} \left[v\cdot\nabla_H\left(\frac{|u|^2}{2}\right)-\left(\int_{-h}^z
  \nabla_H\cdot
  vd\xi\right)\partial_z\left(\frac{|u|^2}{2}\right)\right]\varphi^2dxdydz
  \nonumber\\
  &+\int_{Q_{2r_0}}[(u\cdot\nabla_H)v
  -(\nabla_H\cdot v)u]\cdot u\varphi^2dxdydz\nonumber\\
  &-\int_{Q_{2r_0}}\nabla_HT\cdot u\varphi^2dxdydz=:I_1+I_2+I_3. \label{3.3-1}
  \end{align}
  Recalling that $\varphi$ is independent of $z$, it follow from the Cauchy-Schwarz inequality that
  \begin{align*}
    J:=&\int_{Q_{2r_0}}(\nabla_Hu:\nabla_H(u\varphi^2)+
  \varepsilon\partial_zu\partial_z(u\varphi^2))dxdydz\\
  =&\int_{Q_{2r_0}}[(|\nabla_Hu|^2+
  \varepsilon|\partial_zu|^2)\varphi^2+\nabla_Hu:u\otimes\nabla_H\varphi^2 ]dxdydz\\
  \geq&\int_{Q_{2r_0}}[(|\nabla_Hu|^2+
  \varepsilon|\partial_zu|^2)\varphi^2-2|\nabla_Hu||u|\varphi|\nabla_H\varphi| ]dxdydz\\
  \geq&\frac34\int_{Q_{2r_0}}(|\nabla_Hu|^2+
  \varepsilon|\partial_zu|^2)\varphi^2 dxdydz-C\int_{Q_{2r_0}}|u|^2|\nabla_H\varphi|^2dxdydz\\
  \geq&\frac34\int_{Q_{2r_0}}(|\nabla_Hu|^2+
  \varepsilon|\partial_zu|^2)\varphi^2 dxdydz-\frac{C}{r_0^2}\|u\|_{2,Q_{2r_0}}^2.
  \end{align*}
  Integration by parts, and recalling that $\varphi$ is independent of $z$, one has
  $$
  I_1=\int_{Q_{2r_0}}\frac{|u|^2}{2}v\cdot\nabla_H\varphi^2dxdydz\leq \int_{Q_{2r_0}}|u|^2|v|\varphi|\nabla_H\varphi|dxdydz,
  $$
  and
  \begin{align*}
    I_2=&-\int_{Q_{2r_0}}[(\nabla_H\cdot u)(v\cdot u)\varphi^2+(u\cdot\nabla_H)(u\varphi^2)\cdot v-v\cdot\nabla_H(|u|^2\varphi^2)]dxdydz\\
    \leq&4\int_{Q_{2r_0}}(|u||v||\nabla_Hu|\varphi^2 +|u|^2|v|\varphi|\nabla_H\varphi|)dxdydz.
  \end{align*}
  For $I_3$, it follows from integration by parts and the Cauchy-Schwarz inequality that
  \begin{align*}
    I_3=&\int_{Q_{2r_0}}T\nabla_H\cdot(u\varphi^2)dxdydz\\
    \leq&\int_{Q_{2r_0}}|T|(|\nabla_Hu|\varphi^2+2|u|\varphi| \nabla_H\varphi|)dxdydz\\
    \leq&\frac14\int_{Q_{2r_0}}|\nabla_Hu|^2\varphi^2dxdydz+C\int_{Q_{2r_0}} (|T|^2\varphi^2+|u|^2|\nabla_H\varphi|^2)dxdydz\\
    \leq&\frac14\int_{Q_{2r_0}}|\nabla_Hu|^2\varphi^2dxdydz+C\left( \|T\|_{2,Q_{2r_0}}^2+\frac{\|u\|_{2,Q_{2r_0}}^2}{r_0^2}\right).
  \end{align*}
  Thanks to the estimates on $J$, $I_1, I_2$ and $I_3$, it follows from (\ref{3.3-1}) that
  \begin{align}
  &\frac{d}{dt}\int_{Q_{2r_0}}|u|^2\varphi^2dxdydz+\int_{Q_{2r_0}}(|\nabla_H u|^2+\varepsilon|\partial_zu|^2)\varphi^2dxdydz\nonumber\\
  \leq&
  C\int_{Q_{2r_0}}(|u|^2|v|\varphi|\nabla_H\varphi|+|u||v||\nabla_Hu|\varphi^2) dxdydz
  +C\left(\|T\|_{2,Q_{2r_0}}^2+\frac{\|u\|_{2,Q_{2r_0}}^2}{r_0^2}\right), \label{2}
  \end{align}
for any $t\in(0,\infty)$.

Using the fact that
$|v(x,y,z,t)|\leq\frac{1}{2h}\int_{-h}^h|v(x,y,\xi,t) |d\xi+\int_{-h}^h|\partial_zv(x,y,\xi,t)|d\xi,$
and recalling (\ref{1}), it follows from Lemma \ref{lads}, Proposition \ref{lem3.1}, and using the Cauchy-Schwarz and Young inequalities that, for any $t\in[0,t_0]$,
\begin{align*}
  \int_{Q_{2r_0}}|u|^2|v|\varphi|&\nabla_H\varphi|dxdydz
  \leq C\int_{Q_{2r_0}}\left(\int_{-h}^h\left(\frac{|v|}{h}+|\partial_zv|\right) d\xi\right) |u|^2\varphi|\nabla_H\varphi|dxdydz\nonumber\\
  =&C\int_{D_{2r_0}}\left(\int_{-h}^h\left(\frac{|v|}{h} +|u|\right) d\xi\right)\left(\int_{-h}^h|u|^2d\xi\right)\varphi
  |\nabla_H\varphi|dxdy\nonumber\\
  \leq&\frac{C}{r_0}\left(\frac{\|v\|_{2,Q_{2r_0}}}{h}+ \|u\|_{2,Q_{2r_0}}\right)\|u\|_{2,Q_{2r_0}}\left(\frac{\|u
  \|_{2,Q_{2r_0}}}
  {r_0}+\|\nabla_Hu\|_{2,Q_{2r_0}}\right)\nonumber\\
  \leq&\frac{C}{r_0}(1+\delta_0)\delta_0\left(\frac{\delta_0}{r_0}+\|\nabla_Hu\|_{2,Q_{2r_0}}
  \right)
  \leq C\delta_0\left(\|\nabla_Hu\|_{2,Q_{2r_0}}^2+ \frac{1}{r_0^2}\right),
\end{align*}
here we have used that $\|v\|_2$ is bounded, and
\begin{align*}
\int_{Q_{2r_0}}|u||v|&|\nabla_Hu|\varphi^2dxdydz
  \leq C\int_{Q_{2r_0}}\left( \int_{-h}^h\left(\frac{|v|}{h}+|\partial_zv|\right)d\xi\right) |u||\nabla_Hu|dxdydz\nonumber\\
  =&C\int_{D_{2r_0}}\left(\int_{-h}^h\left(\frac{|v|}{h} +|u|\right)d\xi\right)\left(\int_{-h}^h|u||\nabla_Hu|d\xi
  \right)
  dxdy\nonumber\\
  \leq&C\frac{\|v\|_{6,Q_{2r_0}}}{h}\|u\|_{2,Q_{2r_0}}^{2/3}\left(\frac{\|u\|_{2,Q_{2r_0}}^{1/3}}{r_0^{1/3}}
  +\|\nabla_Hu\|_{2,Q_{2r_0}}^{1/3}\right)\|\nabla_Hu\|_{2,Q_{2r_0}}\nonumber\\
  &+C\|u\|_{2,Q_{2r_0}}\left(\frac{\|u\|_{2,Q_{2r_0}}}{r_0}+\|\nabla_Hu\|_{2,Q_{2r_0}}\right)
  \|\nabla_Hu\|_{2,Q_{2r_0}}\nonumber\\
  \leq&C\left(\frac{\delta_0}{r_0^{1/3}}
  +\delta_0^{2/3}\|\nabla_Hu\|_{2,Q_{2r_0}}^{1/3}+ \frac{\delta_0^2}{r_0} +\delta_0\|\nabla_Hu\|_{2,Q_{2r_0}}\right)\|\nabla_Hu\|_{2,Q_{2r_0}}
  \nonumber\\
  \leq&C\left(\delta_0\|\nabla_Hu\|_{2,Q_{2r_0}}^2+
  \frac{1}{r_0^2}\right).
\end{align*}

Substituting the above two inequalities into (\ref{2}), and using (\ref{1}), one obtains
\begin{align*}
  &\frac{d}{dt}\|u\varphi\|_{2,Q_{2r_0}}^2+ \|\nabla_Hu\varphi\|_{2,Q_{2r_0}}^2 +\varepsilon\|\partial_zu\varphi\|_{2,Q_{2r_0}}^2
  \leq C\delta_0\|\nabla_Hu\|_{2,Q_{2r_0}}^2+\frac{C}{r_0^2},
\end{align*}
and thus
\begin{align*}
  \sup_{0\leq s\leq t} \|u\varphi\|_{2,Q_{2r_0}}^2&+\int_0^t( \|\nabla_Hu\varphi\|_{2,Q_{2r_0}}^2 +\varepsilon\|\partial_zu\varphi\|_{2,Q_{2r_0}}^2) ds\\
  \leq& C\delta_0\int_0^t\|\nabla_Hu\|_{2,Q_{2r_0}}^2ds
  +\frac{C}{r_0}t,
\end{align*}
for any $t\in[0,t_0]$. Recalling that $\varphi\equiv1$ on $D_{r_0}$, this inequality implies
\begin{align*}
   \sup_{0\leq s\leq t}\|u\|_{2,Q_{r_0}}^2+\int_0^t( \|\nabla_Hu\|_{2,Q_{r_0}}^2+\varepsilon\|\partial_zu\|_{2,Q_{r_0}}^2)ds
  \leq C\delta_0\int_0^t\|\nabla_Hu\|_{2,Q_{2r_0}}^2ds
  +\frac{C}{r_0^2}t,
\end{align*}
for any $t\in[0,t_0]$.

Similarly, the above inequality still holds true with $Q_{r_0}$ and $Q_{2r_0}$ replaced by
$Q_{r_0}(\textbf{x}^\text{H})$ and $Q_{2r_0}(\textbf{x}^\text{H})$, respectively. As a result, using the fact that
$$
\sup_{\textbf{x}^\text{H}\in M}\int_0^t\|\nabla_Hu\|_{2,Q_{2r_0}(\textbf{x}^\text{H})}^2ds\leq N\sup_{\textbf{x}^\text{H}\in M} \int_0^t\|\nabla_Hu\|_{2,Q_{r_0}(\textbf{x}^\text{H})}^2ds,
$$
where $N$ (an absolute constant) is the the least number of the disks of radius $r_0$ that covers a disk of radius $2r_0$, we have
\begin{align*}
\sup_{\textbf{x}^\text{H}\in M}&\left[\sup_{0\leq s\leq t}\|u\|_{2,Q_{r_0}(\textbf{x}^\text{H})}^2 \right. +\left.\int_0^t(\|\nabla_Hu\|_{2,Q_{r_0}(\textbf{x}^\text{H})}^2+\varepsilon\|\partial_z
  u\|_{2,Q_{r_0}(\textbf{x}^\text{H})}^2)ds\right]\\
  \leq&CN\delta_0\sup_{\textbf{x}^\text{H}\in M}\left(\int_0^t\|\nabla_Hu\|_{2,Q_{r_0}(\textbf{x}^\text{H})}^2ds\right)+\frac{C}{r_0^2}t,
\end{align*}
and thus, recalling that $\delta_0$ is sufficiently small, there is a positive constant $C_0$ depending only on $h$, such that
\begin{equation}
  \sup_{\textbf{x}^\text{H}\in M}\left(\sup_{0\leq s\leq t}\|u\|_{2,Q_{r_0}(\textbf{x}^\text{H})}^2+
  \int_0^t(\|\nabla_Hu\|_{2,Q_{r_0}(\textbf{x}^\text{H})}^2+\varepsilon\|\partial_z
  u\|_{2,Q_{r_0}(\textbf{x}^\text{H})}^2)ds\right)\leq \frac{C_0}{r_0^2}t\leq 6\delta_0^2, \label{4}
\end{equation}
for any $t\leq\min\left\{t_0,\frac{6\delta_0^2r_0^2}{C_0}\right\}.$

Recalling the definition of $t_0$, one has $t_0\leq1$. If $t_0=1$, then (\ref{4}) holds true for any $t\leq t_0^*$, with
$$
t_0^*:=\min\left\{1,\frac{6\delta_0^2r_0^2}{C_0}\right\}.
$$
If $t_0<1$, then by the definition of $t_0$ and noticing that $u\in C([0,\infty);L^2(\Omega))$, (\ref{4}) implies that
$$
\min\left\{t_0,\frac{6\delta_0^2r_0^2}{C_2}\right\}<t_0,
$$
and thus $t_0>\frac{6\delta_0^2r_0^2}{C_0}\geq t_0^*$. Therefore,
(\ref{4}) still holds true for any $t\leq t_0^*$. In conclusion, we have
estimate (\ref{4}) for any $t\leq t_0^*$. By virtue of estimate (\ref{4}), and covering the domain $M\times(-h,h)$ by finite many $Q_{r_0}$'s, one obtains the estimate stated in the
proposition, and thus completes the proof.
\end{proof}

Now we give the estimates on the horizontal derivatives of the velocity field.

\begin{proposition}
  \label{lem3.3}
  Let $(v,T)$ be as in Proposition \ref{glo}. Then one has
  \begin{align*}
    &\frac{d}{dt}\|\nabla_Hv\|_2^2+(\|\Delta_Hv\|_2^2+\varepsilon\|\nabla_H\partial_zv\|_2^2)\\
    \leq&C(\|v\|_2^2+\|u\|_2^2+1)^2(\|\nabla_Hv\|_2^2+\|\nabla_Hu\|_2^2+1)\|\nabla_Hv\|_2^2,
  \end{align*}
  for any $t\in(0,\infty)$.
\end{proposition}

\begin{proof}
  Multiplying equation (\ref{3.1}) by $-\Delta_Hv$, and integrating
  over $\Omega$, it follows from integrating by parts and using Cauchy-Schwarz and Young inequalities that
  \begin{align}
    &\frac{1}{2}\frac{d}{dt}\int_\Omega |\nabla_Hv|^2dxdydz+\int_\Omega(|\Delta_Hv|^2+\varepsilon |\nabla_H\partial_zv|^2)dxdydz\nonumber\\
    =&\int_\Omega\left[(v\cdot\nabla_H)v-\left(\int_{-h}^z\nabla_H\cdot vd\xi\right)\partial_zv-\nabla_H\left(\int_{-h}^zTd\xi\right)\right]\cdot\Delta_Hvdxdydz\nonumber\\
    \leq&C\int_\Omega\left[|v||\nabla_Hv||\Delta_Hv|
    +\left(\int_{-h}^h|\nabla_Hvd\xi\right)|u||\Delta_Hv|\right]dxdydz\nonumber\\
    &+\frac{1}{4}\|\Delta_Hv\|_2^2+C\|\nabla_HT\|_2^2.\label{3.6}
  \end{align}

  Using the fact that $|\varphi(z)|\leq\frac{1}{2h} \int_{-h}^h|\varphi(\xi)|d\xi+\int_{-h}^h |\partial_z\varphi(\xi)|d\xi$, for every $z\in(-h,h)$, then by Lemma \ref{lad}, and by using the Young inequality, we have the following estimate
  \begin{align*}
    &C\int_\Omega|v||\nabla_Hv||\Delta_Hv|dxdydz\\
    \leq&C\int_M\left(\int_{-h}^h\left(\frac{|v|}{h}+|u|\right)d\xi\right)\left(\int_{-h}^h|\nabla_Hv||\Delta_Hv|d\xi\right)dxdy\\
    \leq&C\left( \frac{\|v\|_2}{h}+\|u\|_2\right)^{1/2}\left( \frac{\|v\|_2}{h} +\|u\|_2+\|\nabla_Hv\|_2+\|\nabla_H u\|_2\right)^{1/2}\\
    &\times\|\nabla_Hv\|_2^{1/2}(\|\nabla_Hv\|_2+\|\nabla_H^2v\|_2)^{1/2}\|\Delta_Hv\|_2\\
    \leq&\frac{1}{8}\|\Delta_Hv\|_2^2+C[\|\nabla_Hv\|_2^2+(\|v\|_2^2+\|u\|_2^2)(\|v\|_2^2+\|u\|_2^2\\
    &+\|\nabla_Hv\|_2^2+\|\nabla_Hu\|_2^2)\|\nabla_Hv\|_2^2]\\
    \leq&\frac{1}{8}\|\Delta_Hv\|_2^2+C(\|v\|_2^2+\|u\|_2^2+1)^2(\|\nabla_Hu\|_2^2+\|\nabla_Hv\|_2^2+1)\|\nabla_Hv\|_2^2.
  \end{align*}
  Applying Lemma \ref{lad} once again, and using the Young inequality, we have
  \begin{align*}
    &C\int_\Omega\left(\int_{-h}^h|\nabla_Hv|d\xi\right)|u||\Delta_Hv|dxdydz\\
    =&C\int_M\left(\int_{-h}^h|\nabla_Hv|d\xi\right)\left(\int_{-h}^h |u||\Delta_Hv|d\xi\right)dxdy\\
    \leq&C\|\nabla_Hv\|_2^{1/2}(\|\nabla_Hv\|_2+\|\nabla_H^2v\|_2)^{1/2}\|u\|_2^{1/2}(\|u\|_2+\|\nabla_Hu\|_2)^{1/2}\|\Delta_Hv\|_2\\
    \leq&\frac{1}{8}\|\Delta_Hv\|_2^2+C[\|\nabla_Hv\|_2^2+\|\nabla_Hv\|_2^2\|u\|_2^2(\|u\|_2^2+\|\nabla_Hu\|_2^2)]\\
    \leq&\frac{1}{8}\|\Delta_Hv\|_2^2+C(\|u\|_2^2+1)^2(\|\nabla_Hu\|_2^2+1)\|\nabla_Hv\|_2^2.
  \end{align*}

  Substituting the above two inequalities into (\ref{3.6}) yields
  \begin{align*}
    &\frac{d}{dt}\|\nabla_Hv\|_2^2+(\|\Delta_Hv\|_2^2+\varepsilon\|\nabla_H\partial_zv\|_2^2)\\
    \leq&C(\|v\|_2^2+\|u\|_2^2+1)^2(\|\nabla_Hv\|_2^2+\|\nabla_Hu\|_2^2+1)\|\nabla_Hv\|_2^2,
  \end{align*}
  proving the conclusion.
\end{proof}

Estimates on the derivatives of the temperature is stated in the following proposition.

\begin{proposition}
  \label{lem3.4}
  Let $(v,T)$ be as in Proposition \ref{glo}. Then it holds that
  \begin{align*}
  &\frac{d}{dt}\|\nabla T\|_2^2+\|\nabla_H\nabla T\|_2^2+\varepsilon\|\partial_z\nabla T\|_2^2\\
  \leq&C(\|v\|_2^2+\|\nabla v\|_2^2+1)^2(\|\nabla_Hv\|_2^2+\|\nabla_H\nabla v\|_2^2+1)(\|\nabla T\|_2^2+1),
  \end{align*}
  for any $t\in(0,\infty)$.
\end{proposition}

\begin{proof}
  Multiplying equation (\ref{3.3}) by $-\Delta T$ and integrating over $\Omega$, then it follows from integrating by parts that
  \begin{align}
    &\frac{1}{2}\frac{d}{dt}\int_\Omega|\nabla T|^2dxdydz+\int_\Omega(|\nabla_H\nabla T|^2+\varepsilon|\partial_z\nabla T|^2)dxdydz\nonumber\\
    =&\int_\Omega\left[v\cdot\nabla_HT-\left(\int_{-h}^z\nabla_H\cdot vd\xi\right)\left(\partial_zT+\frac1h\right)\right]\Delta Tdxdydz\nonumber\\
    =&\int_\Omega\left[v\cdot\nabla_HT-\left(\int_{-h}^z\nabla_H\cdot vd\xi\right)\left(\partial_zT+\frac1h\right)\right]\Delta_H Tdxdydz\nonumber\\
    &+\int_\Omega\left[v\cdot\nabla_HT-\left(\int_{-h}^z\nabla_H\cdot vd\xi\right)\left(\partial_zT+\frac1h\right)\right]\partial_z^2 Tdxdydz\nonumber\\
    =&\int_\Omega\left[v\cdot\nabla_HT-\left(\int_{-h}^z\nabla_H\cdot vd\xi\right)\left(\partial_zT+\frac1h\right)\right]\Delta_H Tdxdydz\nonumber\\
    &-\int_\Omega\left[(\partial_zv\cdot\nabla_HT+v\cdot\nabla_H\partial_zT)\partial_zT-\frac{1}{2}
    (\nabla_H\cdot v)|\partial_zT|^2+\frac{\nabla_H\cdot v}{h}\partial_zT\right]dxdydz\nonumber\\
    =&\int_\Omega\left[v\cdot\nabla_HT-\left(\int_{-h}^z\nabla_H\cdot vd\xi\right)\left(\partial_zT+\frac1h\right)\right]\Delta_H Tdxdydz\nonumber\\
    &-\int_\Omega(u\cdot\nabla_HT\partial_zT+2v\cdot\nabla_H\partial_zT \partial_zT+\frac1h\nabla_H\cdot v\partial_zT)dxdydz\nonumber\\
    \leq&\int_\Omega\left[v\cdot\nabla_HT-\left(\int_{-h}^z\nabla_H\cdot vd\xi\right)\left(\partial_zT+\frac1h\right)\right]\Delta_H Tdxdydz\nonumber\\
    &-\int_\Omega(u\cdot\nabla_HT\partial_zT+2v\cdot\nabla_H\partial_zT \partial_zT)dxdydz+C\|\nabla_Hv\|_2\|\partial_zT\|_2.\label{3.7}
  \end{align}
  Using the fact that $|v(z)|\leq\frac{1}{2h}\int_{-h}^h|v(\xi)|d\xi +\int_{-h}^h|\partial_zv(\xi)|d\xi$, for all $z\in(-h,h)$, it follows from Lemma \ref{lad}, and using the Young inequality that
  \begin{align*}
    &\left|\int_\Omega(v\cdot\nabla_HT\Delta_HT-2v\cdot\nabla_H\partial_zT\partial_zT)dxdydz\right|\\
    \leq&C\int_M\left(\int_{-h}^h\left(\frac{|v|}{h}+|u|\right) d\xi\right)\left(\int_{-h}^h|\nabla T||\nabla_H\nabla T|d\xi\right)dxdy\\
    \leq&C\left(\frac{\|v\|_2}{h}+\|u\|_2 \right)^{1/2}\left(\frac{\|v\|_2}{h} +\|u\|_2+\|\nabla_Hv\|_2+\|\nabla_Hu\|_2\right)^{1/2}\\
    &\times\|\nabla T\|_2^{1/2}(\|\nabla T\|_2+\|\nabla_H\nabla T\|_2)^{1/2}\|\nabla_H\nabla T\|_2\\
    \leq&\frac{1}{6}\|\nabla_H\nabla T\|_2^2+C[\|\nabla T\|_2^2+(\|v\|_2^2+\|u\|_2^2)(\|v\|_2^2+\|u\|_2^2\\
    &+\|\nabla_Hv\|_2^2+\|\nabla_Hu\|_2^2)\|\nabla T\|_2^2]\\
    \leq&\frac{1}{6}\|\nabla_H\nabla T\|_2^2+C(\|v\|_2^2+\|u\|_2^2+1)^2(\|\nabla_Hv\|_2^2+\|\nabla_Hu\|_2^2+1)\|\nabla T\|_2^2.
  \end{align*}
  Recalling that $T$ is odd and periodic in $z$, one has
  $T(x,y,-h,t)=-T(x,y,h,t)=-T(x,y,-h,t)$, and thus $T|_{z=-h,h}=0$,
  which implies $\nabla_HT|_{z=-h,h}=0$. Thanks to this fact, it
  follows $|\nabla_HT(z)|\leq\int_{-h}^h|\nabla_H\partial_zT|d\xi$, for
  every $z\in(-h,h)$.
  Using this inequality, it follows from Lemma \ref{lad} and the Young
  inequality that
  \begin{align*}
    &\left|\int_\Omega u\cdot\nabla_HT\partial_zTdxdydz\right|\\
    \leq&C\int_M\left(\int_{-h}^h|\nabla_H\partial_zT|d\xi\right)\left(\int_{-h}^h|u||\partial_zT|d\xi\right)dxdy\\
    \leq&C\|\nabla_H\partial_zT\|_2\|u\|_2^{1/2}(\|u\|_2+\|\nabla_H u\|_2)^{1/2}\|\partial_zT\|_2^{1/2}(\|\partial_zT\|_2+\|\nabla_H\partial_zT\|_2)^{1/2}\\
    \leq&\frac{1}{6}\|\nabla_H\partial_zT\|_2^2+C[\|\partial_zT\|_2^2+\|u\|_2^2(\|u\|_2^2+\|\nabla_Hu\|_2^2)\|\partial_zT\|_2^2]\\
    \leq&\frac{1}{6}\|\nabla_H\partial_zT\|_2^2+C(\|u\|_2^2+1)^2(\|\nabla_Hu\|_2^2+1)\|\partial_zT\|_2^2.
  \end{align*}
  Applying Lemma \ref{lad} once again, and using the Young inequality yields
  \begin{align*}
    &\left|\int_\Omega\left(\int_{-h}^z\nabla_H\cdot vd\xi\right)\partial_zT\Delta_HTdxdydz\right|\\
    \leq&C\int_M\left(\int_{-h}^h|\nabla_Hv|d\xi\right)\left(\int_{-h}^h|\partial_zT||\Delta_HT|d\xi\right)dxdy\\
    \leq&C\|\nabla_Hv\|_2^{1/2}(\|\nabla_Hv\|_2+\|\nabla_H^2v\|_2)^{1/2}\|\partial_zT\|_2^{1/2}(\|\partial_zT\|_2
    +\|\nabla_H\partial_zT\|_2)^{1/2}\|\Delta_HT\|_2\\
    \leq&\frac{1}{6}\|\nabla_H\nabla T\|_2^2+C[\|\partial_zT\|_2^2+\|\nabla_Hv\|_2^2(\|\nabla_Hv\|_2^2+\|\Delta_Hv\|_2^2)\|\partial_zT\|_2^2]\\
    \leq&\frac{1}{6}\|\nabla_H\nabla T\|_2^2+C(\|\nabla_Hv\|_2^2+1)^2(\|\Delta_Hv\|_2^2+1)\|\partial_zT\|_2^2.
  \end{align*}
  Substituting these inequalities into (\ref{3.7}) leads to
  \begin{align*}
  &\frac{d}{dt}\|\nabla T\|_2^2+\|\nabla_H\nabla T\|_2^2 +\varepsilon\|\partial_z\nabla T\|_2^2\\
  \leq&C(\|v\|_2^2+\|\nabla v\|_2^2+1)^2(\|\nabla_Hv\|_2^2+\|\nabla_H\nabla v\|_2^2+1)(\|\nabla T\|_2^2+1),
  \end{align*}
  completing the proof.
\end{proof}

\begin{proposition}
  \label{lem3.5}
  Let $(v,T)$ be as in Proposition \ref{glo}.
Then we have
  \begin{align*}
  \|\partial_tv\|_2^2+\|\partial_tT\|_2^2\leq& C[\varepsilon^2(\|\partial_zu\|_2^2+\|\partial_z^2T\|_2^2)+(\|v\|_{H^1}^2+\|T
  \|_{H^1}^2+1)^2\\
  &\times(\|\nabla_Hv\|_{H^1}^2+\|\nabla_HT\|_{H^1}^2+1)],
  \end{align*}
  for any $t\in(0,\infty)$.
\end{proposition}

\begin{proof}
Define functions $f_1$ and $f_2$ as
\begin{align*}
  f_1=&-(v\cdot\nabla_H)v+\left(\int_{-h}^z\nabla_H\cdot v(x,y,\xi,t)d\xi\right)\partial_zv\\
  &-f_0k\times v+\nabla_H\left(\int_{-h}^zT(x,y,\xi,t)d\xi\right)
\end{align*}
and
\begin{equation*}
  f_2=v\cdot\nabla_HT+\left(\int_{-h}^z\nabla_H\cdot v(x,y,\xi,t)d\xi\right)\left(\partial_zT+\frac1h\right).
\end{equation*}

Applying Lemma \ref{lad}, and using the Sobolev and Young inequalities, we have the following estimates
\begin{align*}
  \|f_2\|_2^2\leq & C\int_\Omega\left[|v|^2|\nabla_HT|^2+ \left(\int_{-h}^h|\nabla_Hv|d\xi\right)^2(1+|\partial_zT|^2)\right]
  dxdydz\\
  \leq&C(\|v\|_4^2\|\nabla_HT\|_4^2+\|\nabla_Hv\|_2^2) +C\int_M\left(\int_{-h}^h|\nabla_Hv|^2d\xi\right)\left(\int_{-h}^h
  |\partial_zT|^2d\xi\right)dxdy\\
  \leq&C(\|v\|_{H^1}^2\|\nabla_HT\|_{H^1}^2+\|\nabla_Hv\|_2^2)\\
  &+C\|\nabla_H
  v\|_2(\|\nabla_Hv\|_2+\|\nabla_H^2v\|_2)\|\partial_zT\|_2(\|\partial_zT\|_2+\|\nabla_H
  \partial_zT\|_2)\\
  \leq&C(\|v\|_{H^1}^2\|\nabla_HT\|_{H^1}^2+\|v\|_{H^1}^2) +C[\|v\|_{H^1}^2(\|v\|_{H^1}^2+\|\nabla_Hv
  \|_{H^1}^2)\\
  &+\|T\|_{H^1}^2(\|T\|_{H^1}^2+\|\nabla_HT\|_{H^1}^2)]\\
  \leq&C(\|v\|_{H^1}^2+\|T\|_{H^1}^2+1)^2(\|\nabla_Hv\|_{H^1}^2+\|\nabla_HT\|_{H^1}^2+1),
\end{align*}
and similarly
\begin{align*}
  \|f_1\|_2^2\leq&C(\|v\|_{H^1}^2+1)^2(\|\nabla_Hv\|_{H^1}^2+1)+C(\|v\|_2^2+\|\nabla_HT\|_2^2)\\
  \leq&C(\|v\|_{H^1}^2+\|T\|_{H^1}^2+1)^2(\|\nabla_Hv\|_{H^1}^2+\|\nabla_HT\|_{H^1}^2+1).
\end{align*}

Note that $v$ and $T$ satisfies
\begin{eqnarray}
  &&\partial_tv-\Delta_Hv-\varepsilon\partial_z^2v+\nabla_Hp_s(x,y,t)=f_1,\label{n1}\\
  &&\int_{-h}^h\nabla_H\cdot v(x,y,\xi,t)d\xi=0,\label{n2}\\
  &&\partial_tT-\Delta_HT-\varepsilon\partial_z^2T=f_2.\label{n3}
\end{eqnarray}
By (\ref{n3}), we have
\begin{align*}
  &\|\partial_tT\|_2^2\leq\|\Delta_HT\|_2^2+\varepsilon^2\|\partial_z^2T\|_2^2 +\|f_2\|_2^2\\
  \leq&C(\|v\|_{H^1}^2+\|T\|_{H^1}^2+1)^2(\|\nabla_Hv\|_{H^1}^2 +\|\nabla_HT\|_{H^1}^2+1)+\varepsilon^2\|\partial_z^2T\|_2^2.
\end{align*}
By the aid of (\ref{n1}) and (\ref{n2}), one can easily see that
$$
\Delta_Hp_s(x,y,t)=\frac{1}{2h}\int_{-h}^h\nabla_H\cdot f_1(x,y,\xi,t)d\xi,
$$
and thus, by the elliptic estimates, one obtains
\begin{align*}
  \|\nabla_Hp_s\|_{L^2(M)}^2\leq&C\left\|\int_{-h}^hf_1(x,y,\xi,t)d\xi\right\|_{L^2(M)}^2\leq C\|f_1\|_2^2\\
  \leq&C(\|v\|_{H^1}^2+\|T\|_{H^1}^2+1)^2(\|\nabla_Hv\|_{H^1}^2+\|\nabla_HT\|_{H^1}^2+1).
\end{align*}
On account of this, it follows from (\ref{n1}) that
\begin{align*}
  \|\partial_tv\|_2^2\leq&C(\|\Delta_Hv\|_2^2+\varepsilon^2 \|\partial_z^2v\|_2^2+\|\nabla_Hp_s\|_2^2+\|f_1\|_2^2)\\
  \leq&C[\varepsilon^2\|\partial_zu\|_2^2+(\|v\|_{H^1}^2+\|T\|_{H^1}^2+1)^2(\|\nabla_Hv\|_{H^1}^2+\|\nabla_HT\|_{H^1}^2+1)].
\end{align*}
Therefore, we have
\begin{align*}
  \|\partial_tv\|_2^2+\|\partial_tT\|_2^2\leq& C[\varepsilon^2(\|\partial_zu\|_2^2+\|\partial_z^2T\|_2^2)+(\|v\|_{H^1}^2+\|T
  \|_{H^1}^2+1)^2\\
  &\times(\|\nabla_Hv\|_{H^1}^2+\|\nabla_HT\|_{H^1}^2+1)],
  \end{align*}
completing the proof.
\end{proof}

By the aid of these propositions, we are now ready to give the proof of local well-posedness part of Theorem \ref{thm1}.

\begin{proof}[\textbf{Proof of local well-posedness}] Consider the
periodic functions $v_0^\varepsilon,
T_0^\varepsilon\in H^2(\Omega)$, such that they are even and odd in $z$, respectively, and
$$
(v_0^\varepsilon,T_0^\varepsilon)\rightarrow(v_0,T_0),\quad\mbox{ as }\varepsilon\rightarrow0,\quad\mbox{in }H^1(\Omega).
$$
It is obvious that $u_0^\varepsilon\rightarrow u_0$, as $\varepsilon\rightarrow0$, in $L^2(\Omega)$, where $u_0^\varepsilon=\partial_zv_0^\varepsilon$ and $u_0=\partial_zv_0$.

Let $\delta_0$ be the constant sated in Proposition \ref{lem3.2}. By the absolutely continuity of the integral, there is a positive number $r_0\leq1$, such that
$$
\sup_{\textbf{x}^\text{H}\in M}\int_{-h}^h\int_{D_{2r_0}(\textbf{x}^\text{H})} |u_0 (x,y,z)|^2dxdydz\leq\frac{\delta_0^2}{2}.
$$
Thus, there exits $\varepsilon_0>0$, depending on $\delta_0$, such that for every $\varepsilon\in(0,\varepsilon_0)$, one has
$$
\sup_{\textbf{x}^\text{H}\in M}\int_{-h}^h\int_{D_{2r_0}(\textbf{x}^\text{H})}|u_0^\varepsilon(x,y,z)|^2dxdydz\leq\delta_0^2.
$$

For any $\varepsilon\in(0,\varepsilon_0)$, let
$(v_\varepsilon,T_\varepsilon)$ be the unique strong solution
corresponding to the initial data $(v_0^\varepsilon, T_0^\varepsilon)$
as stated by Proposition \ref{glo}. Set
$u_\varepsilon=\partial_zv_\varepsilon$. By Proposition \ref{lem3.1} and Proposition \ref{lem3.2}, we have the estimate
\begin{align}
\sup_{0\leq t\leq t_0^*}&(\|v_\varepsilon\|_2^2+\|T_\varepsilon\|_2^2+\|u_\varepsilon\|_2^2)+\int_0^{t_0^*}
(\|\nabla_Hv_\varepsilon\|_2^2\nonumber\\
&+\varepsilon|\partial_zv_\varepsilon|^2+\|\nabla_HT_\varepsilon\|_2^2+\|\nabla_Hu_\varepsilon\|_2^2+\varepsilon
\|\partial_zu_\varepsilon\|_2^2)ds\leq C, \label{3.13-1}
\end{align}
where $C$ is a positive constant and $t_0^*$ is the same constant as in
Proposition \ref{lem3.2}, both depend on $\delta_0$, but are independent of $\varepsilon\in(0,\varepsilon_0)$. On account of the above estimate, by using the Gronwall inequality, one can easily obtain from Proposition \ref{lem3.3}, Proposition \ref{lem3.4} and Proposition \ref{lem3.5} that
\begin{align*}
\sup_{0\leq t\leq t_0^*}&(\|\nabla_Hv_\varepsilon\|_2^2+\|\nabla T_\varepsilon\|_2^2)+\int_0^{t_0^*}(\|\Delta_Hv_\varepsilon\|_2^2\\
  &+\varepsilon\|\nabla_H\partial_zv_\varepsilon\|_2^2+\|\nabla_H\nabla T_\varepsilon\|_2^2
  +\|\partial_tv_\varepsilon\|_2^2+\|\partial_tT_\varepsilon\|_2^2)\leq C,
\end{align*}
which, combined with (\ref{3.13-1}), gives
\begin{align}
\sup_{0\leq t\leq t_0^*}&(\|v_\varepsilon\|_{H^1}^2+\| T_\varepsilon\|_{H^1}^2)+\int_0^{t_0^*}(\|\nabla_Hv_\varepsilon\|_{H^1}^2 \nonumber\\
  &+\varepsilon\|\partial_zv_\varepsilon\|_{H^1}^2+\|\nabla_H T_\varepsilon\|_{H^1}^2
  +\|\partial_tv_\varepsilon\|_2^2+\|\partial_tT_\varepsilon\|_2^2)\leq C,
  \label{estuniform}
\end{align}
for a positive constant $C$, which depends on $\delta_0$, but is
independent of $\varepsilon\in(0,\varepsilon_0)$. Thanks to the above
estimate, by the Aubin-Lions lemma, i.e. Lemma \ref{AL}, there is a
subsequence, still denoted by $(v_\varepsilon,T_\varepsilon)$, and
$(v,T)$, such that as $\varepsilon\rightarrow0$ one has
\begin{eqnarray*}
  &(v_\varepsilon, T_\varepsilon)\rightarrow(v,T),\quad\mbox{ in }C([0,t_0^*];L^2(\Omega)),\\
  &\varepsilon\partial_zv_\varepsilon\rightarrow0,\quad\mbox{ in }L^\infty(0,t_0^*;L^2(\Omega)),\\
  &(v_\varepsilon,T_\varepsilon)\overset{*}{\rightharpoonup}(v,T),\quad\mbox{ in }L^\infty(0,t_0^*;H^1(\Omega)),\\
  &(\nabla_Hv_\varepsilon,\nabla_HT_\varepsilon)\rightharpoonup(\nabla_Hv,\nabla_HT),\quad\mbox{ in }L^2(0,t_0^*;H^1(\Omega)),\\
  &(\partial_tv_\varepsilon,\partial_tT_\varepsilon)\rightharpoonup(\partial_tv,\partial_tT),\quad
  \mbox{ in }L^2(0,t_0^*;L^2(\Omega)),
\end{eqnarray*}
where $\rightharpoonup$ and $\overset{*}{\rightharpoonup}$ are the weak convergence and weak-$*$ convergence, respectively. Thanks to these convergences, noticing that
\begin{eqnarray*}
  &\left(\int_{-h}^z\nabla_Hv_\varepsilon d\xi\right)\partial_zv_\varepsilon= \partial_z\left(\left(\int_{-h}^z\nabla_H\cdot v_\varepsilon d\xi\right) v_\varepsilon\right)-(\nabla_H\cdot v_\varepsilon) v_\varepsilon,\\
  &\left(\int_{-h}^z\nabla_Hv_\varepsilon d\xi\right)\partial_zT_\varepsilon= \partial_z\left(\left(\int_{-h}^z\nabla_H\cdot v_\varepsilon d\xi\right) T_\varepsilon\right)-(\nabla_H\cdot v_\varepsilon) T_\varepsilon,
\end{eqnarray*}
one can take the limit $\varepsilon\rightarrow0$, at the level of the
subsequence, to system (\ref{3.1})--(\ref{3.3}) to conclude that $(v,T)$ is a strong solution to system (\ref{1.5})--(\ref{1.9}) on $\Omega\times(0,t_0^*)$.

Recalling the regularity properties of the strong solutions, the uniqueness of strong solutions to system (\ref{1.5})--(\ref{1.9}) is a direct corollary of Proposition 2.4 of \cite{CAOLITITI2}. This completes the proof of the local well-posedness part of Theorem \ref{thm1}.
\end{proof}

\section{Global existence of strong solutions}
\label{sec4}
In this section, we prove that if the initial data $(v_0, T_0)\in H^1$
has the additional regularity that $\partial_z v_0\in L^m(\Omega)$, for some $m\in(2,\infty)$, and $(v_0, T_0)\in L^\infty(\Omega)$, then the local strong solution established in the previous section can be extended to be a global one, in other words, we will prove the global existence part of Theorem \ref{thm1}.

Checking the proof in the previous section, to prove the global
existence of strong solutions to system (\ref{1.5})--(\ref{1.9}), it
suffices to establish estimate (\ref{estuniform}), for any finite
time interval $[0,\mathcal T)$,
for the global strong solution $(v,T)$ to system
(\ref{3.1})--(\ref{3.3}), subject to (\ref{1.8})--(\ref{1.9}).
Moreover, by Propositions \ref{lem3.1}, \ref{lem3.4} and \ref{lem3.5},
to get such an estimate, we only need to prove the
$L^\infty(0,\mathcal T; L^2(\Omega))$ estimate for $\partial_zv$,
for any finite time interval $[0,\mathcal T)$.

The following proposition is a straightforward corollary of Proposition 3.1 in \cite{CAOLITITI3}.

\begin{proposition}
  \label{prop5.1}
Let $(v,T)$ be as in Proposition \ref{glo}.
Then, for any finite time $\mathcal T$, we have the estimate
$$
\sup_{0\leq t\leq\mathcal T}\sup_{2\leq q<\infty}\frac{\|v\|_q}{\sqrt q} \leq C,
$$
where $C$ is a positive constant depending only on $h, \mathcal T$ and  $\|(v_0,T_0)\|_\infty$, but is independent of $\varepsilon$.
\end{proposition}

\begin{proposition}
  \label{prop5.2}
Let $(v,T)$ be as in Proposition \ref{glo}, and set $u:=\partial_zv$.
Then, for any finite time $\mathcal T$, we have the following estimates
$$
  \frac{d}{dt}\|u\|_q^q+\int_\Omega|u|^{q-2}(|\nabla_Hu|^2+\varepsilon |\partial_zu|^2)dxdydz\leq C(1+\|v\|_\infty^2)(1+\|u\|_q^q),
$$
and
\begin{align*}
  &\frac{d}{dt}\|\nabla_Hv\|_2^2+\|\Delta_Hv\|_2^2+\varepsilon\|\nabla_H\partial_z v\|_2^2\\
\leq& C\|v\|_\infty^2\|\nabla_Hv\|_2^2+C(\|u\|_r^{ \frac{4r}{r-2}}+\|\nabla_HT\|_2^2+1),
\end{align*}
for any $q\in[2,\infty), r\in(2,\infty)$ and any $t\in(0,\mathcal T)$, where $C$ is a positive constant depending only on $q, r, h, \mathcal T, \|v_0\|_2$ and $\|T_0\|_\infty$, but is independent of $\varepsilon$.
\end{proposition}

\begin{proof}
The first conclusion has been proved in Proposition 4.1 (i) of \cite{CAOLITITI3}. We now prove the second one.
Multiplying equation (\ref{3.1}) by $-\Delta_Hv$, and integrating over $\Omega$, then it follows from integrating by parts that
\begin{align*}
  &\frac{1}{2}\frac{d}{dt}\int_\Omega|\nabla_Hv|^2dxdydz+\int_\Omega\Big(
  |\Delta_Hv|^2+\varepsilon|\nabla_H\partial_zv|^2\Big)dxdydz\nonumber\\
  =&\int_\Omega\left[(v\cdot\nabla_H)v-\left(\int_{-h}^z\nabla_H\cdot vd\xi\right)u-\nabla_H\left(\int_{-h}^zTd\xi\right)\right]\cdot\Delta_Hvdxdydz\nonumber\\
  \leq&C(\|v\|_\infty\|\nabla_Hv\|_2+\|\nabla_HT\|_2\|)\Delta_Hv\|_2
  +\int_M\int_{-h}^h|\nabla_Hv|dz\int_{-h}^h|u||\Delta_Hv|dzdxdy.
\end{align*}
Recalling that $\sup_{0\leq t\leq\mathcal T}\|v\|_2^2\leq C$, guaranteed by Proposition \ref{lem3.1}, it follows from the H\"older, Minkowski and Gagliardo-Nirenberg inequalities that
\begin{align*}
  &\int_M\int_{-h}^h|\nabla_Hv|dz\int_{-h}^h|u||\Delta_Hv|dzdxdy\\
  \leq&\int_M\int_{-h}^h|\nabla_Hv|dz\left(\int_{-h}^h|u|^2dz\right)^{\frac{1}{2}} \left(\int_{-h}^h|\Delta_Hv|^2dz\right)^{\frac12}dxdy\\
  \leq&\left[\int_M\left(\int_{-h}^h|\nabla_Hv|dz\right)^{\frac{2r} {r-2}}dxdy\right] ^{\frac{r-2}{2r}}\left[\int_M\left(\int_{-h}^h|u|^2dz\right)^{\frac{r}{2}} dxdy\right]^{\frac{1}{r}}\|\Delta_Hv\|_2\\
  \leq&\int_{-h}^h\|\nabla_Hv\|_{\frac{2r}{r-2}, M}dz\left(\int_{-h}^h\|u\|_{r,M}^2dz\right)^{\frac12}\|\Delta_Hv\|_2\\
  \leq&C\int_{-h}^h\|v\|_{2,M}^{\frac12-\frac1r}(\|v\|_{2,M}+\|\Delta_Hv\|_{2,M}) ^{\frac12+\frac1r} dz\|u\|_r\|\Delta_Hv\|_2\\
  \leq&C(\|v\|_2+\|v\|_2^{\frac12-\frac1r}\|\Delta_Hv\|_2^{\frac12+\frac1r}) \|u\|_r\|\Delta_Hv\|_2\\
  \leq&C(1+\|\Delta_Hv\|_2^{\frac12+\frac1r}) \|u\|_r\|\Delta_Hv\|_2.
\end{align*}
Substitute the above inequality into the previous one and using the Young inequality, we then get
\begin{align*}
  &\frac{1}{2}\frac{d}{dt}\|\nabla_Hv\|_2^2+\|\Delta_Hv\|_2^2 +\varepsilon\|\nabla_H \partial_zv\|_2^2\\
  \leq&C(\|v\|_\infty\|\nabla_Hv\|_2+\|\nabla_HT\|_2\|)\Delta_Hv\|_2 +C(1+\|\Delta_Hv\|_2^{\frac12+\frac1r}) \|u\|_r\|\Delta_Hv\|_2\\
  \leq&\frac 12 \|\Delta_Hv\|_2^2+C(\|v\|_\infty^2\|\nabla_Hv\|_2^2+\|u\|_r^{ \frac{4r}{r-2}}+\|\nabla_HT\|_2^2+1),
\end{align*}
which implies the conclusion.
\end{proof}

Thanks to the above two propositions, we can apply the logarithmic type Sobolev embedding inequality (Lemma \ref{logsob}) and the logarithmic type Gronwall inequality (Lemma \ref{loggron}) to derive the $L^\infty(0,\mathcal T; L^2(\Omega))$ estimate on $\nabla v$.

\begin{proposition}
  \label{prop5.3}
  Let $(v,T)$ be as in Proposition \ref{glo}, and let $m\in(2,\infty)$.
  Then, for any finite time $\mathcal T$, we have
  $$
  \sup_{0\leq t\leq\mathcal T}(\|\nabla v\|_2^2+\|\partial_zv\|_m^m)+\int_0^\mathcal T(\|\nabla_H\nabla v\|_2^2+\varepsilon\|\partial_z\nabla v\|_2^2)dt\leq C,
  $$
  for a positive constant $C$ depending only on $m, h, \mathcal T$ and $\|v_0\|_{H^1}+\|\partial_zv_0\|_m+\|(v_0,T_0)\|_\infty$, but is independent of $\varepsilon$.
\end{proposition}

\begin{proof}
  Given $\mathcal T\in(0,\infty)$. Set $u=\partial_zv$, and define
  \begin{eqnarray*}
    &&A_1(t)=\|u(t)\|_2^2+\|u(t)\|_m^m+e,\quad B_1(t)=\|\nabla_Hu(t)\|_2^2+\varepsilon\|\partial_zu\|_2^2+e,\\
    &&A_2(t)=\|\nabla_Hv(t)\|_2^2+e,~\qquad\qquad B_2(t)=\|\Delta_Hv(t)\|_2^2+\varepsilon\|\partial_z\nabla_Hv\|_2^2+e.
  \end{eqnarray*}
  By Proposition \ref{prop5.2}, we have
  \begin{eqnarray*}
    &&\frac{d}{dt}A_1(t)+B_1(t)\leq C(1+\|v(t)\|_\infty^2)A_1(t),\\
    &&\frac{d}{dt}A_2(t)+B_2(t)\leq C\|v(t)\|_\infty^2A_2(t)+CA_1^\lambda(t)+C\|\nabla_HT(t)\|_2^2,
  \end{eqnarray*}
  for any $t\in(0,\mathcal T)$, where $\lambda=\frac{4}{m-2}$, and $C$ is a positive constant depending only on $m, h, \mathcal T, \|v_0\|_2$ and $\|T_0\|_\infty$, but independent of $\varepsilon$.

  Multiplying the first inequality by $1+\lambda A_1^{\lambda-1}(t)$, and summing the resulting inequality with the second one yields
  \begin{align*}
\frac{d}{dt}(A_1(t)&+A_1^\lambda(t)+A_2(t))+(1+\lambda A_1^{\lambda-1}(t))B_1(t)+B_2(t)\\
    \leq&C(1+\|v\|_\infty^2)(A_1(t)+A_1^\lambda(t)+A_2(t))+C\|\nabla_HT(t)\|_2^2,
  \end{align*}
  for any $t\in(0,\mathcal T)$, where $C$ is a positive constant depending only on $m, h, \mathcal T, \|v_0\|_2$ and $\|T_0\|_\infty$, and is independent of $\varepsilon$.

  Summing both sides of the above inequality with $A_1(t)+A_1^\lambda(t)+A_2(t)$, and setting
  \begin{eqnarray*}
    &&A(t)=A_1(t)+A_1^\lambda(t)+A_2(t),\quad B(t)=A_1(t)+B_1(t)+B_2(t),\\
    &&g(t)=1+\|v(t)\|_\infty^2,\qquad f(t)=C\|\nabla_HT(t)\|_2^2,
  \end{eqnarray*}
  we have
  \begin{equation*}
    \frac{d}{dt}A(t)+B(t)\leq Cg(t)A(t)+f(t).
  \end{equation*}
We are going to show that
\begin{equation}\label{N5.3}
g(t)\leq C\log(e+B(t)),
\end{equation}
and thus
\begin{equation*}
    \frac{d}{dt}A(t)+B(t)\leq K A(t)\log(e+B(t))+f(t),
\end{equation*}
for a positive constant $K$ depending only on $h,\mathcal T$ and $\|v_0\|_\infty+\|T_0\|_\infty$.
Noticing that $\|f\|_{L^1((0,\mathcal T))}\leq C$, for a constant $C$ depending only on $h,\mathcal T$ and $\|v_0\|_2+\|T_0\|_2$, the conclusion follows from the logarithmic type Gronwall inequality, i.e. Lemma \ref{loggron}.

We still need to verify (\ref{N5.3}). By Proposition \ref{prop5.1}, Lemma \ref{logsob} and the Sobolev and Poincar\'e inequalities, we have
\begin{align*}
  \|v(t)\|_\infty^2\leq&C\max\left\{1, \sup_{q\geq2}\frac{\|v\|_q^2}{q}\right\} \log(e+\|\nabla_Hv\|_6+\|v\|_6+\|u\|_2+\|v\|_2)\\
  \leq&C\log(e+\|\nabla_Hv\|_{H^1}+\|v\|_{H^1}+\|u\|_2)\\
  \leq&C\log(e+\|\nabla_Hv\|_2+\|\nabla\nabla_Hv\|_2+\|u\|_2)\\
  \leq&C\log(e+\|\Delta_Hv\|_2+\|\nabla_Hu\|_2+\|u\|_2)\leq C\log(e+B(t)),
\end{align*}
verifying (\ref{N5.3}). This completes the proof.
\end{proof}

We are now ready to prove the global existence part of Theorem \ref{thm1}.

\begin{proof}[\textbf{Proof of global existence}] Let $j_\varepsilon$ be the standard modifier, and set $v_0^\varepsilon=v_0*j_\varepsilon$ and $T_0^\varepsilon=T_0*j_\varepsilon$. Then we have
\begin{eqnarray*}
  &(v_0^\varepsilon, T_0^\varepsilon)\rightarrow(v_0, T_0),\mbox{ in }H^1(\Omega),\quad\mbox{and}\quad\partial_zv_0^\varepsilon \rightarrow\partial_zv_0,\mbox{ in }L^m(\Omega),\\
  &\|v_0^\varepsilon\|_{H^1}\leq\|v_0\|_{H^1},\quad \|\partial_zv_0^\varepsilon\|_m\leq\|\partial_zv_0\|_m, \quad\mbox{and}\quad\|v_0^\varepsilon\|_\infty\leq\|v_0\|_\infty,\\
  &\|T_0^\varepsilon\|_{H^1}\leq\|T_0\|_{H^1}, \quad\mbox{and}\quad\|T_0^\varepsilon\|_\infty\leq\|T_0\|_\infty.
\end{eqnarray*}
Let $(v_\varepsilon, T_\varepsilon)$ be the unique global strong solution to system (\ref{3.1})--(\ref{3.3}), subject to (\ref{1.8})--(\ref{1.9}), with initial data $(v_0^\varepsilon, T_0^\varepsilon)$, as stated in Proposition \ref{glo}.

By Proposition \ref{prop5.3}, for any $\mathcal T\in(0,\infty)$, there is a positive constant $C$ depending only on $h, \mathcal T$ and $\|v_0\|_{H^1}+\|\partial_zv_0\|_m+\|v_0\|_\infty+\|T_0\|_\infty$, but independent of $\varepsilon$, such that
$$
\sup_{0\leq t\leq\mathcal T}(\|\nabla_Hv_\varepsilon\|_2+\|\partial_zv_\varepsilon\|_2^2+\|\partial_z v_\varepsilon\|_m^m)+\int_0^\mathcal T(\|\nabla\nabla_Hv_\varepsilon\|_2^2+\varepsilon\|\partial_z\nabla v_\varepsilon\|_2^2)dt\leq C.
$$
Thanks to this estimates, by Propositions \ref{lem3.1}, \ref{lem3.4} and \ref{lem3.5}, it follows from the Gronwall inequality that
\begin{align*}
  &\sup_{0\leq t\leq\mathcal T}(\|v_\varepsilon\|_{H^1}^2+\| T_\varepsilon\|_{H^1}^2)+\int_0^{\mathcal T}(\|\nabla_Hv_\varepsilon\|_{H^1}^2 +\|\nabla_H T_\varepsilon\|_{H^1}^2
  +\|\partial_tv_\varepsilon\|_2^2+\|\partial_tT_\varepsilon\|_2^2)\leq C,
\end{align*}
from which, the same argument as in the proof of local well-posedness at the end of section \ref{sec3} yields the global existence of strong solutions.
This completes the proof.
\end{proof}

\section{Appendix: a logarithmic Sobolev inequality}

In this appendix, we prove a logarithmic Sobolev inequality for anisotropic Sobolev functions, that is the following:

\begin{lemma}
  \label{lemapp}
Let $\textbf{p}=(p_1, p_2,\cdots, p_N)$, with $p_i\in(1,\infty)$, and
$$
\sum_{i=1}^N\frac{1}{p_i}<1.
$$
Then we have
\begin{align*}
  \|F\|_{L^\infty(\mathbb R^N)}\leq& C_{N,\textbf{p},\lambda}\max\left\{1,\sup_{r\geq2} \frac{\|F\|_{L^r(\mathbb R^N)}}{r^\lambda}\right\}\\
  &\times\log^\lambda
  \left(\sum_{i=1}^N(\|F\|_{L^{p_i}(\mathbb R^N)}+\|\partial_iF\|_{L^{p_i}(\mathbb R^N)})+e\right),
\end{align*}
for any $\lambda>0$.
\end{lemma}

\begin{proof}
We only give the detail proof for the case that the spatial dimension $N\geq3$, the case that $N=2$ can be given similarly. Without loss of generality, we can suppose that
$|F(0)|=\|F\|_\infty$. Let $\phi\in C_0^\infty(B_1)$, with $\phi\equiv1$ on $B_{1/2}$, and set $f=F\phi$. Set
$$
\alpha_i=\frac{1}{p_i}\in(0,1),\quad i=1,2,\cdots,N,
$$
and introduce the new variable $y=(y_1,\cdots,y_N)$, with
$$
y_i=|x_i|^{\alpha_i-1}x_i,\quad i=1,2,\cdots,N.
$$

Taking $\mu_i$ and $\kappa_i$ as
\begin{eqnarray*}
&&\mu_i =\frac{1+\sum_{j=1}^N\alpha_j}{1+\sum_{j=1}^N\alpha_j-2\alpha_i},\qquad\kappa_i =\frac{p_i\left(1+\sum_{j=1}^N\alpha_j\right)}{1-\sum_{j=1}^N\alpha_j},
\end{eqnarray*}
then it is obvious that $\mu_i>1$ and $\kappa_i>p_i$, and one can easily check that
\begin{equation}
  \label{mukappa}
  \frac{1}{\mu_i}+\frac{1}{\kappa_i}+\frac{1}{p_i}=1.
\end{equation}
Setting $\alpha=\sum_{j=1}^N\alpha_j$, then we have
\begin{align*}
  &\left(\sum_{j=1}^N\alpha_j-\alpha_i\right)\mu_i-\sum_{j=1}^N\alpha_j
  = \left( \alpha -\alpha_i\right)\frac{1+\alpha}{1+ \alpha -2\alpha_i}-\alpha
  = \frac{\alpha_i(\alpha-1)}{1+\alpha-2\alpha_i},
\end{align*}
from which, noticing that
$$
\alpha=\sum_{j=1}^N\alpha_j=\sum_{j=1}^N\frac{1}{p_j}<1,\quad \mbox{and}\quad 1+\alpha-2\alpha_i>1-\alpha_i>0,
$$
we have
\begin{equation}\label{mukappa1}
  N+\left(\sum_{j=1}^N\alpha_j-\alpha_i\right)\mu_i-\sum_{j=1}^N\alpha_j<N.
\end{equation}

Set $K_1=\{(y_1,\cdots,y_N)|y_i=|x_i|^{\alpha_i-1}x_i, x\in B_1\}$.
Recall that $f$ can be represented in terms of $\Delta f$ by the Newtonian potential. Recalling (\ref{mukappa}) and (\ref{mukappa1}), it follows from integration by parts and the H\"older inequality that, for any $q\geq3$,
\begin{align}
  |f(0)|^q=&C_N\left|\int_{\mathbb R^N}\frac{1}{|x|^{N-2}}\Delta\left(|f|^q (|x_1|^{\alpha_1-1}x_1,\cdots,|x_N|^{\alpha_N-1}x_N)\right)
  dx\right|\nonumber\\
  =&C_N\left|\int_{B_1}\frac{1}{|x|^{N-2}}\Delta \left(|f|^q(|x_1|^{\alpha_1-1}x_1,\cdots,|x_N|^{\alpha_N-1}x_N)\right)
  dx\right|\nonumber\\
  =&C_N\left|\int_{B_1}\sum_{i=1}^N\partial_{x_i} \left(\frac{1}{|x|^{N-2}}\right)\partial_{x_i}\left(|f|^q(y)\right)
  dx\right|\nonumber\\
  =&(N-2)C_Nq\sum_{i=1}^N\alpha_i\left|\int_{B_1} \frac{|x_i|^{\alpha_i-1}x_i}{|x|^N}|f|^{q-1}(y)|\partial_{y_i}
  f(y)|dx\right|\nonumber\\
  =&(N-2)C_N\left(\Pi_{j=1}^N\alpha_j\right)^{-1}\nonumber\\
  &\times q\sum_{i=1}^N\alpha_i
  \left|\int_{K_1}\frac{|x_i|^{\alpha_i-1}x_i}{|x|^N} \Pi_{j=1}^N|x_j|^{1-\alpha_j}|f|^{q-1}(y)|\partial_{y_i}
  f(y)|dy\right|\nonumber\\
  \leq&C_{N,\textbf{p}}q\sum_{i=1}^N\|f\|_{(q-1)\kappa_i}^{q-1}\|\partial_if\|_{p_i}
  \left[\int_{K_1}\left(\frac{|x_i|^{\alpha_i}}{|x|^N}\Pi_{j=1}^N|x_j|^{1-\alpha_j}\right)^{\mu_i}dy
  \right]^{\frac{1}{\mu_i}}\nonumber\\
  \leq&C_{N,\textbf{p}}q\sum_{i=1}^N\|f\|_{(q-1)\kappa_i}^{q-1}\|\partial_if\|_{p_i}
  \left[\int_{B_1}\left(\frac{|x_i|^{\alpha_i}}{|x|^N}\right)^{\mu_i}
  \left(\Pi_{j=1}^N|x_j|^{1-\alpha_j}\right)^{\mu_i-1}dx
  \right]^{\frac{1}{\mu_i}}\nonumber\\
  \leq&C_{N,\textbf{p}}q\sum_{i=1}^N\|f\|_{(q-1)\kappa_i}^{q-1}\|\partial_if\|_{p_i}
  \left(\int_{B_1}\frac{dx} {|x|^{N+\left(\sum_{j=1}^N\alpha_j-\alpha_i\right)\mu_i-\sum_{j=1}^N
  \alpha_j}}\right)^{\frac{1}{\mu_i}}\nonumber\\
  \leq&C_{N,\textbf{p}}q\sum_{i=1}^N \|f\|_{(q-1)\kappa_i}^{q-1}\|\partial_if\|_{p_i}.\label{new}
\end{align}

With the aid of (\ref{new}), for any $q\geq3$, noticing that $q^{\frac{1}{q}}\leq C$ and $(q-1)\kappa_i\geq2$, we deduce
\begin{align*}
  |f(0)|\leq&C_{N,\textbf{p}}\sum_{j=1}^N\|f\|_{(q-1)\kappa_i}^{1-\frac{1}{q}}\|\partial_i
  f\|_{p_i}^{\frac{1}{q}}\\
  =&C_{N,\textbf{p}}\sum_{i=1}^N\left[\frac{\|f\|_{(q-1)\kappa_i}}{((q-1)\kappa_i)^\lambda}
  \right]^{1-\frac{1}{q}}[(q-1)\kappa_i]^{\lambda\left(1-\frac{1}{q}\right)}\|\partial_i
  f\|_{p_i}^{\frac{1}{q}}\\
  \leq&C_{N,\textbf{p},\lambda}\sum_{i=1}^N\left[\frac{\|f\|_{(q-1)\kappa_i}}{((q-1)\kappa_i)
  ^\lambda}
  \right]^{1-\frac{1}{q}}q^{\lambda}\|\partial_i
  f\|_{p_i}^{\frac{1}{q}}\\
  \leq&C_{N,\textbf{p},\lambda}\max\left\{1,\sup_{r\geq2}\frac{\|f\|_r}{r^\lambda}\right\}q^\lambda\sum
  _{i=1}^N\|\partial_i
  f\|_{p_i}^{\frac{1}{q}},
\end{align*}
and thus
\begin{equation*}
  |f(0)|\leq C_{N,\textbf{p},\lambda}\max\left\{1,\sup_{r\geq2}\frac{\|f\|_r}{r^\lambda}\right\}\sum
  _{i=1}^N\inf_{q\geq3}
  \left(q^\lambda(\|\partial_i
  f\|_{p_i}+e^{4\lambda})^{\frac{1}{q}}\right).
\end{equation*}

One can check that
$$
\log(\|\partial_if\|_{p_i}+e^{4\lambda})\leq\max\{1,4\lambda\}\log(\|\partial_if\|_{p_i}+e)
$$
and
$$
\inf_{q\geq3}
  \left(q^\lambda(\|\partial_i
  f\|_{p_i}+e^{4\lambda})^{\frac{1}{q}}\right)=\left(\frac{e}{\lambda}\right)^\lambda\log^\lambda
  (\|\partial_if\|_{p_i}+e^{4\lambda}).
$$
Therefore, we have
\begin{equation*}
  |f(0)|\leq C_{N,\textbf{p},\lambda}\max\left\{1,\sup_{r\geq2}\frac{\|f\|_r}{r^\lambda}\right\}\sum
  _{i=1}^N\log^\lambda
  (\|\partial_if\|_{p_i}+e).
\end{equation*}
This implies
\begin{align*}
\|F\|_\infty=&|f(0)|\leq C_{N,\textbf{p},\lambda}\max\left\{1,\sup_{r\geq2}\frac{\|f\|_r}{r^\lambda}\right\}\sum
  _{i=1}^N\log^\lambda
  (\|\partial_if\|_{p_i}+e)\\
  \leq&C_{N,\textbf{p},\lambda}\max\left\{1,\sup_{r\geq2}\frac{\|F\|_r}{r^\lambda}\right\}\sum
  _{i=1}^N\log^\lambda
  (\|\partial_iF\|_{p_i}+\|F\|_{p_i}+e)\\
  \leq&C_{N,\textbf{p},\lambda}\max\left\{1,\sup_{r\geq2}\frac{\|F\|_r}{r^\lambda}\right\}
  \log^\lambda
  (\|F\|_{W^{1,\textbf{p}}}+e),
\end{align*}
proving the conclusion.
\end{proof}

\section*{Acknowledgments}
{This work was supported in part by the ONR grant N00014-15-1-2333 and the NSF grants DMS-1109640 and DMS-1109645.}
\par

\end{document}